\documentclass[11pt,reqno]{amsart}

\usepackage[a-1b]{pdfx}
\usepackage[utf8]{inputenc}
\usepackage[T1]{fontenc}
\usepackage{tikz-cd}

\usepackage{amssymb}
\usepackage{amsthm}
\usepackage{comment}
\usepackage{esint}
\usepackage{tikz}
\usepackage{fourier}

\usepackage{xcolor}
\hypersetup{
    colorlinks,
    linkcolor={red!50!black},
    citecolor={blue!50!black},
    urlcolor={blue!80!black}
}

\usepackage{pgfornament}

\renewcommand{\H}{\mathbb{H}}
\newcommand{\N}{\mathbb{N}}
\newcommand{\R}{\mathbb{R}}
\newcommand{\Z}{\mathbb{Z}}
\newcommand{\Hz}{\mathbb{H}_\Z}

\newcommand{\cF}{\mathcal{F}}

\newcommand{\length}{\ell}

\newtheorem{thm}{Theorem}[section]

\newtheorem{lemma}[thm]{Lemma}
\newtheorem{lem}[thm]{Lemma}
\newtheorem{prop}[thm]{Proposition}
\newtheorem{cor}[thm]{Corollary}

\newcommand{\bt}{\begin{thm}}
\newcommand{\et}{\end{thm}}
\newcommand{\bc}{\begin{cor}}
\newcommand{\ec}{\end{cor}}
\newcommand{\bl}{\begin{lem}}
\newcommand{\el}{\end{lem}}
\newcommand{\bp}{\begin{prop}}
\newcommand{\ep}{\end{prop}}
\newtheorem{defn}[thm]{Definition}
\newcommand{\bd}{\begin{defn}}      
\newcommand{\ed}{\end{defn}}
\newtheorem{quest}[thm]{Question}
\newcommand{\bq}{\begin{quest}}
\newcommand{\eq}{\end{quest}}
\theoremstyle{remark}

\newcommand{\br}{\begin{rmrk}}
\newcommand{\er}{\end{rmrk}}

\DeclareMathOperator{\id}{id}
\DeclareMathOperator{\diam}{diam}
\DeclareMathOperator{\Lip}{Lip}
\DeclareMathOperator{\inter}{int}

\DeclareMathOperator{\vol}{vol}
\DeclareMathOperator{\avol}{avol}
\DeclareMathOperator{\Ar}{Ar}

\DeclareMathOperator{\supp}{supp}
\DeclareMathOperator{\FV}{FV}

\newcommand{\from}{\colon}
\date{April 1, 2021}
\title{Constructing H\"older maps to Carnot groups}

\begin{document}
\bibliographystyle{plain}

\author{Stefan Wenger}

\address
  {Department of Mathematics\\ University of Fribourg\\ Chemin du Mus\'ee 23\\ 1700 Fribourg, Switzerland}
\email{stefan.wenger@unifr.ch}

\author{Robert Young}

\address{Courant Institute of Mathematical Sciences\\
  New York University\\
  251 Mercer St.\\
  New York, NY  10012\\
  USA}
\email{ryoung@cims.nyu.edu}

\thanks{S.~W.~was partially supported by Swiss National Science Foundation Grants 165848 and 182423.  R.~Y.~was supported by a Sloan Research Fellowship and by National Science Foundation grants 1612061 and 2005609. Parts of this paper were written while R.~Y.~was a visiting member at the Institute for Advanced Study, supported by NSF grant 1926686. }

\begin{abstract}
In this paper, we construct H\"older maps to Carnot groups equip\-ped with a Carnot metric, especially the first Heisenberg group \(\H\). Pansu and Gromov \cite{Gro-CC-96} observed that any surface embedded in $\H$ has Hausdorff dimension at least \(3\), so there is no \(\alpha\)--H\"older embedding of a surface into \(\H\) when \(\alpha>\frac{2}{3}\).  Z\"ust \cite{Zus15} improved this result to show that when \(\alpha>\frac{2}{3}\), any \(\alpha\)--H\"older map from a simply-connected Riemannian manifold to $\H$ factors through a metric tree.  In the present paper, we show that Z\"ust's result is sharp by constructing \((\frac{2}{3}-\epsilon)\)--H\"older maps from \(D^2\) and \(D^3\) to $\H$ that do not factor through a tree.  We use these to show that if $0<\alpha < \frac{2}{3}$, then the set of $\alpha$--H\"older maps from a compact metric space to $\H$ is dense in the set of continuous maps and to  construct proper degree--1 maps from \(\R^3\) to \(\H\) with H\"older exponents arbitrarily close to \(\frac{2}{3}\).  
\end{abstract}

\maketitle

\section{Introduction and statement of results}

The first Heisenberg group $\H$, equipped with a Carnot metric, is a subriemannian manifold.  The Hausdorff dimension of such a manifold is greater than its topological dimension; the Heisenberg group, for instance, has topological dimension \(3\) and Hausdorff dimension \(4\).  It follows that there is no surjective Lip\-schitz map from \(\R^3\) to \(\H\), since Lipschitz maps cannot increase Hausdorff dimension.  Indeed, the image of the $3$--dimensional unit ball \(D^3\) under an \(\alpha\)--H\"older map has Hausdorff dimension at most \(\frac{3}{\alpha}\), so when \(\alpha > \frac{3}{4}\), there is no \(\alpha\)--H\"older map from $D^3$ to $\H$ whose image contains a metric ball.

When \(\alpha<\frac{3}{4}\), a construction like that of Kaufman \cite{Kau79} can be used to construct an \(\alpha\)--H\"older map from \(D^3\) to \(\H\) whose image contains a ball, but when \(\frac{2}{3}<\alpha<\frac{3}{4}\), the topology of such maps is very restricted.  These conditions arise from the fact, proved by Gromov in \cite[0.6.C, 2.1]{Gro-CC-96} and therein also attributed to Pansu, that any surface embedded in \(\H\) has topological dimension \(2\) but Hausdorff dimension at least \(3\), so if \(\alpha>\frac{2}{3}\), then the image of a surface under an \(\alpha\)--H\"older map cannot be a surface.  Indeed, Z\"ust \cite{Zus15} showed that if \(M\) is a simply-connected Riemannian manifold and \(f\from M\to \H\) is \(\alpha\)--H\"older with \(\alpha>\frac{2}{3}\), then \(f\) factors through a metric tree.  Moreover, Le Donne and Z\"ust \cite{LDZ13} proved that if $\alpha>\frac{1}{2}$ then any $\alpha$--H\"older surface in $\H$ (if it exists) must intersect many vertical lines in a topological Cantor set. 

In \cite[0.5.C]{Gro-CC-96}, Gromov asked: 
\begin{quote}
Given two [Carnot--Carathéodory] spaces $V$ and $W$ and a real number $0<\alpha\le 1$, describe the space of $C^\alpha$--maps $f\from W\to V$. For example, when can each continuous map $W\to V$ be uniformly approximated by $C^\alpha$--maps? When can $W$ be $C^\alpha$--embed\-ded into $V$? When are $V$ and $W$ $C^\alpha$ homeomorphic? etc.
\end{quote}
 The special case of finding the maximum $\alpha$ such that there is an $\alpha$--H\"older homeomorphism from $\R^3$ to $\H$ has become known as the H\"older equivalence problem.
 It follows from the results of Pansu and Gromov \cite[2.1]{Gro-CC-96} or Z\"ust \cite{Zus15}, both mentioned above, that if $\alpha>\frac{2}{3}$ then there is no locally $\alpha$--H\"older homeomorphism from $\R^3$ to $\H$ and that continuous maps from $\R^3$ to $\H$ cannot be approximated by $\alpha$--H\"older maps.
On the other hand, a smooth or $C^2$ map from $\R^3$ to $\H$ is locally $\frac{1}{2}$--H\"older, so a continuous map from $\R^3$ to $\H$ can be approximated by a locally $\frac{1}{2}$--H\"older map, and there are many locally $\frac{1}{2}$--H\"older homeomorphisms from $\R^3$ to $\H$.

 In this paper, we will partially answer Gromov's question by showing that there are many $\alpha$--Hölder maps from $\R^n$ to $\H$ for $\alpha$ arbitrarily close to $\frac{2}{3}$, including maps that are topologically nontrivial (e.g., proper and degree--$1$) and maps that approximate arbitrary continuous functions. Our constructions build on techniques developed in \cite{WY-LipHom} and \cite{LWY16}.
Results like this were first suggested by unpublished work of Piotr Hajłasz, Jake Mirra, and Armin Schikorra, who explored constructing Hölder maps by numerical methods and found results pointing to the possible existence of nontrivial surfaces in $\H$ with Hölder exponent larger than $\frac{1}{2}$ \cite{MirraComm}.

Our first result provides H\"older extensions of maps from subsets of $\R^2$ to general Carnot groups equipped with a Carnot metric. In order to state our theorem we recall the following definition. A pair $(X,Y)$ of metric spaces $X$ and $Y$ is said to have the $\alpha$--H\"older extension property, $0<\alpha\leq 1$, if there exists $L\geq 1$ such that for every subset $Z\subset X$ every $(\lambda, \alpha)$--H\"older map $\varphi\colon Z\to Y$ has an $(L\lambda, \alpha)$--H\"older extension $\overline{\varphi}\colon X\to Y$.

\bt\label{thm:Hoelder-ext-Carnot-Dehn-step}
 Let $G$ be a Carnot group of step $k$, endowed with a Carnot metric $d_c$. Then the pair $(\R^2, (G, d_c))$ has the $\alpha$--H\"older extension property for every $\alpha<\frac{2}{k+1}$.
\et

In particular, given a closed Lipschitz curve $\gamma\from S^1\to \H$, we can extend $\gamma$ to a $\alpha$--Hölder map of a disc for any $\alpha<\frac{2}{3}$. Our construction produces a disc which is not even locally an embedding, even if $\gamma$ is an embedding, and it is an open question (see \cite[0.5.D]{Gro-CC-96}) whether there are $\alpha$--Hölder embeddings from $\R^2$ to $\H$ for $\frac{1}{2}<\alpha<\frac{2}{3}$.

We can extend the construction used in the theorem above to produce H\"older maps from $3$--dimensional Riemannian manifolds to \(\H\). Let $d_R$ be the distance coming from a left-invariant Riemannian metric on $\H$ and let $d_c$ be the associated Carnot metric. Let $\H_\Z$ be the integer lattice in $\H$.

\begin{thm}\label{thm:3d equivariant}
  For any $\alpha<\frac{2}{3}$ there is a locally $\alpha$--H\"older map $g\from (\H,d_R) \to (\H,d_c)$ which is $\Hz$--equivariant (that is, $g(h\cdot x)=h\cdot g(x)$ for all $h\in \Hz$ and $x\in \H$) and equivariantly homotopic to the identity.
\end{thm}

As a consequence, we obtain the following approximation result.
\begin{thm}\label{thm:approximations}
  Let $Y$ be a compact metric space and let $0<\alpha < \frac{2}{3}$.  Any continuous map $\varphi \from Y \to \H$ can be approximated uniformly by $\alpha$-H\"older maps.  
\end{thm}
This exponent is sharp; when $\frac{2}{3}<\alpha\le 1$, \cite{Zus15} implies that any $\alpha$-H\"older map $\psi\from D^2 \to \H$ factors through a metric tree, so $\psi(\partial D^2)$ has filling radius zero.  A curve with nonzero filling radius cannot be uniformly approximated by curves with filling radius zero, so if $\varphi\from D^2\to \H$ is a continuous map such that $\varphi(\partial D^2)$ is a simple closed curve, then $\varphi$ cannot be uniformly approximated by $\alpha$--H\"older maps when $\alpha>\frac{2}{3}$.

The map constructed in Theorem~\ref{thm:3d equivariant} is self-similar, and by taking a tangent cone at a carefully-chosen point, we furthermore obtain the following result. Recall that a continuous map between metric spaces is called proper if preimages of compact sets are compact. A Euclidean similarity is a composition of a scaling, translation, and rotation/reflection. A Heisenberg similarity is a scaling composed with a left-trans\-lation.

\begin{thm}\label{thm:3d tangent cone}
  For any $\epsilon>0$, there is an $0<\epsilon'<\epsilon$ such that there is a globally $(\frac{2}{3}-\epsilon')$--H\"older map $F\from \R^3\to (\H,d_c)$ which is proper and of degree $1$.  This map is self-similar in the sense that there is a Euclidean similarity $h\from \R^3\to \R^3$ and a Heisenberg similarity $m\from \H\to \H$, both with scaling factors bigger than $1$, such that $m(F(h^{-1}(x)))=F(x)$ for all $x\in \R^3$.  
\end{thm}

Gromov showed \cite[3.1.A]{Gro-CC-96} that $\alpha$--Hölder maps from $\R^3$ to $\H$ must have local degree zero when $\alpha>\frac{2}{3}$ and asked whether this exponent can be improved; this construction shows that Gromov's result is sharp.

We now give an outline of the proofs of our results.  In order to prove Theorem~\ref{thm:Hoelder-ext-Carnot-Dehn-step} it is enough to show that there exists $L\geq 1$ such that every $\lambda$-Lipschitz curve $\gamma\colon S^1\to (G, d_c)$ admits an $(L\lambda, \alpha)$--H\"older extension $f\colon D^2\to (G, d_c)$ to the $2$--dimensional unit ball $D^2$ (see \cite[Theorem 6.4]{LWY16}).

We construct such an extension using methods based on the H\"older extension results in \cite{LWY16}.  The main ingredient is the so-called coarse Dehn function, also known as Gromov's mesh function, which is roughly defined as follows. Let $X$ be a geodesic metric space and $\varepsilon>0$. For $r>0$ the coarse Dehn function $\Ar_{X, \varepsilon}(r)$ is the smallest number such that any closed curve in $X$ of length at most $r$ can be subdivided into $\Ar_{X, \varepsilon}(r)$ closed curves of length at most $\varepsilon$. We refer to Section~\ref{sec:prelims} for a precise definition.  When $X$ is a Carnot group $G$ of step $s$, equipped with a Carnot metric $d_c$, it can be shown that
\begin{equation}\label{eq:intro-bound-Ar}
\Ar_{(G, d_c), \frac{r}{n}}(r) \leq C\cdot n^{s+1}
\end{equation}
(see Lemma~\ref{lem:bound-coarse-Dehn-Carnot-step}).

Fix a closed Lipschitz curve $\gamma$ in $(G, d_c)$ of length $r$. Using the bound \eqref{eq:intro-bound-Ar} we can construct nested subdivisions of $\gamma$ as follows. Let $n\in\N$ be sufficiently large, only depending on $\alpha$ and $C$, and set $N=C n^{s+1}$. Then there exist closed curves \(\gamma^1_{1},\dots, \gamma^1_{N}\) of length \(n^{-1} r\) that subdivide $\gamma$.  Each curve \(\gamma^1_j\) can then be subdivided into curves \(\gamma^2_{(j-1)N+1},\dots, \gamma^2_{jN}\) of length \(n^{-2} r\), and so on.  We will construct the H\"older extension \(f\) of $\gamma$ so that its image is the closure of the union $\bigcup_{i,j} \gamma^i_j$.  

We start by constructing a family of nested discs in $D^2$. Let $B^0_1=D^2$ and let $B^1_1,\dots, B^1_N\subset B^0_1$ be disjoint discs of equal radius $\rho$, where $\rho\approx N^{-\frac{1}{2}}$. 
We repeat the process on each disc; for each $i\ge 0$ and $j=1,\dots, N^i$, we choose $N$ disjoint discs of radius $\rho^{i+1}$ inside $B^i_j$ and label them $B^{i+1}_{(j-1)N+1},\dots, B^{i+1}_{jN}$. Let $M_i=D^2\setminus \bigcup_j \inter B^i_j$, so that 
$$S^1=M_0\subset M_1\subset M_2\subset \dots \subset D^2$$
is an increasing sequence of subsets and $\bigcup_i M_i$ is the complement of a Cantor set $K$.

Next, we define \(f\) on the boundaries of the discs so that \(f(\partial D^2)=\gamma\) and \(f(\partial B^i_j)=\gamma^i_j\) for every \(i\) and \(j\).
We extend $f$ to the rest of $D^2$ by noting that the complement \(D^2\setminus \bigcup_{i,j} \partial B^i_j\) consists of the Cantor set \(K=\bigcap(D^2\setminus M_i)\) and infinitely many connected components that are each homeomorphic to a genus $0$ surface with \(N+1\) boundary components.  Let 
$$S^{i}_{j}=B^{i}_{j}\setminus \bigcup_{m=(j-1)N+1}^{jN} \inter B^{i+1}_m$$
be one such component.
Then $f$ sends the outer boundary $\partial B^i_{j}$ to $\gamma^i_j$ and the $N$ inner boundary components to $N$ curves $\gamma^{i+1}_m$ that subdivide $\gamma^i_j$. 
Consequently, we can extend \(f\) over \(S^i_j\) so that
$$f(S^i_j)\subset \bigcup_{m=(j-1) N+1}^{jN} \gamma^{i+1}_m$$
and \(f\) is Lipschitz on \(S^i_j\).
This defines \(f\) on \(D^2\setminus K\).  
%Each $\partial B^i_j$ is a closed curve of length $\approx \rho^{i}$, where $\rho\approx \sqrt{L}\approx n^{\frac{s+1}{2}}$. 
%Since $f(\partial B^i_j)$ is a closed curve of length $\approx n^{-i}$, this map is at most $\frac{\log n}{-\log \rho}$--Hölder.
If we construct the extensions to the $S^i_j$'s carefully, we can ensure that \(f\) is H\"older on $D^2\setminus K$ and extend \(f\) continuously to \(K\)  to obtain the desired map.

Note that \(f\) is far from injective.  In fact, for any neighborhood \(U\) of \(K\), the image \(f(D^2\setminus U)\) has Hausdorff and topological dimension 1; actually, the restriction \(f|_{D^2\setminus U}\) factors through a graph.

% what's my point? really, my point is kind of that this sort of thing is common. that it's not really hard to construct maps like this, that it's really about finer and finer subdivisions of a space, that there's a fair bit of flexibility in what we're doing. hrm.
\medskip 

This construction uses two main ideas: First, we can reduce the problem of constructing a Hölder extension of a closed curve $\gamma$ to the subproblem of constructing Hölder extensions of each curve in a subdivision of $\gamma$. That is, we can extend $\gamma$ by subdividing $\gamma$ into $\gamma_1,\dots,\gamma_N$, constructing a map from an $N$--holed disc $M$ to $G$ that sends the outer boundary of $M$ to $\gamma$ and the boundaries of the holes to the $\gamma_i$'s, and constructing Hölder extensions of the $\gamma_i$'s.
Second, we don't need to actually solve any of the subproblems. As long as we can split the original problem into smaller and smaller subproblems, we can pass to a limit where all of them disappear. % The Cat in the Hat reveals Little Cat A in his hat, whose hat contains Little Cats B, C, D, and so on. As ever-smaller cats appear, they divide the original spot into more and more spots. Finally, when the whole house seems stained beyond repair, Little Cat Z appears, too microscopic to see, and unleashes a VOOM that instantly cleans the house.

We use a version of these ideas to construct the map from $(\H,d_R)$ to $(\H,d_c)$ in Theorem~\ref{thm:3d equivariant}, but the extra dimension adds some complications. To simplify matters, consider constructing a map $g$ from the Euclidean ball $D^3$ to $(\H,d_c)$. 
As in the two-dimensional case, one can construct a sequence of nested sets \(M_1\subset M_2\subset \dots\subset D^3\),  where \(M_i\) consists of $D^3$ with \(k_i\) balls \(B^i_1,\dots, B^i_{k_i}\) removed.
One could attempt to construct a Hölder map by following the outline of Theorem~\ref{thm:Hoelder-ext-Carnot-Dehn-step}. That is, each ball \(B^i_j\) contains some collection of smaller balls \(C_{1},\dots, C_{m}\). If $g$ is already defined on the outer sphere $\partial B^i_j$, one can extend it to $B^i_j\setminus \bigcup_k C_k$ so that it sends the inner spheres to a subdivision $g(\partial C_1),\dots, g(\partial C_m)$ of $g(\partial B^i_j)$.
The images \(g(\partial M_i)\) then form a sequence of finer and finer subdivisions of \(g(\partial D^3)\), and we can extend the map to all of $D^3$ by passing to a limit.

The main difficulty with this outline is that the spheres and their subdivisions need to be parametrized by Hölder maps. These are more difficult to construct than the horizontal subdivisions of curves that we used in the two-dimensional case.
That is, in the two-dimensional case, we reduced from a horizontal curve $\gamma$ to its horizontal subdivision $\gamma_1,\dots, \gamma_N$ by constructing a map from a $N$--holed disc $M$ to the horizontal graph $\bigcup_i \gamma_i$ that sends the outer boundary to $\gamma$ and the inner boundaries to the $\gamma_i$. Since the edges of the graph are horizontal, the map on $M$ can be taken to be Lipschitz and thus Hölder.
Suppose instead that $S$ is a sphere in $\H$, subdivided into spheres $S_1,\dots, S_N$. Then there are continuous maps from the $N$--holed ball $M$ to $\bigcup S_i$ that send the outer boundary to $S$ and the inner boundaries to the $S_i$'s, but those maps may not be Hölder, even if all of the $S_i$'s are images of Hölder spheres.

This is difficult to solve directly, so in Sections~\ref{sec:cellulation-and-Q}--\ref{sec:holder-bounds}, we develop a different approach to constructing Hölder maps. We give a brief sketch.
We start by choosing a $0<s<1$ and constructing a sequence of simplicial complexes $X_i$. We scale the metric on  $X_i$ so that each simplex is a regular simplex with sides of length $s^i$ and ask that the $X_i$ approximate $(\H,d_c)$ more and more closely in the sense that there are $C>1$ and bilipschitz homeomorphisms $\iota_i\from X_i\to (\H,d_R)$ that satisfy
$$C^{-1} d_{X_i}(x,y) - Cs^{i} \le d_{c}(\iota_i(x),\iota_i(y))\le C d_{X_i}(x,y) + Cs^{i}$$
for all $x,y\in X_i$.

We construct a map $P\from X_0\to \H$ by composing a sequence of Lipschitz cellular maps $Q_i\from X_i\to X_{i+1}$ that are \emph{admissible maps}.
An admissible map is a cellular map such that for each $d>0$ and each $d$--cell $\delta\in \cF^d(X_i)$, there is a collection of closed balls $B^\delta_1,\dots, B^\delta_{l_\delta}\subset \delta$ with disjoint interiors such that $Q_i$ sends each $B^\delta_j$ homeomorphically to a $d$--cell of $X_{i+1}$ and collapses $\delta\setminus \bigcup_j B^\delta_j$ to a lower-dimensional set, i.e.,
$$Q_i(\delta\setminus \bigcup_j B^\delta_j)\subset X_{i+1}^{(d-1)},$$ where $X_{i+1}^{(d-1)}$ denotes the $(d-1)$--skeleton of $X_{i+1}$.
We call the $B^\delta_i$'s \emph{uncollapsed balls}, and we construct $Q_i$ so that the $B^\delta_i$'s are pairwise disjoint, as in Figure~\ref{fig:admissible}. 

\begin{figure}
\begin{center}
\def\svgwidth{\textwidth}
%% Creator: Inkscape 1.0.2 (e86c8708, 2021-01-15), www.inkscape.org
%% Accompanies image file 'admissible.pdf' (pdf, eps, ps)
\begingroup%
  \makeatletter%
  \providecommand\color[2][]{%
    \errmessage{(Inkscape) Color is used for the text in Inkscape, but the package 'color.sty' is not loaded}%
    \renewcommand\color[2][]{}%
  }%
  \providecommand\transparent[1]{%
    \errmessage{(Inkscape) Transparency is used (non-zero) for the text in Inkscape, but the package 'transparent.sty' is not loaded}%
    \renewcommand\transparent[1]{}%
  }%
  \providecommand\rotatebox[2]{#2}%
  \newcommand*\fsize{\dimexpr\f@size pt\relax}%
  \newcommand*\lineheight[1]{\fontsize{\fsize}{#1\fsize}\selectfont}%
  \ifx\svgwidth\undefined%
    \setlength{\unitlength}{390.59441757bp}%
    \ifx\svgscale\undefined%
      \relax%
    \else%
      \setlength{\unitlength}{\unitlength * \real{\svgscale}}%
    \fi%
  \else%
    \setlength{\unitlength}{\svgwidth}%
  \fi%
  \global\let\svgwidth\undefined%
  \global\let\svgscale\undefined%
  \makeatother%
  \begin{picture}(1,0.30047332)%
    \lineheight{1}%
    \setlength\tabcolsep{0pt}%
    \put(0,0){\includegraphics[width=\unitlength,page=1]{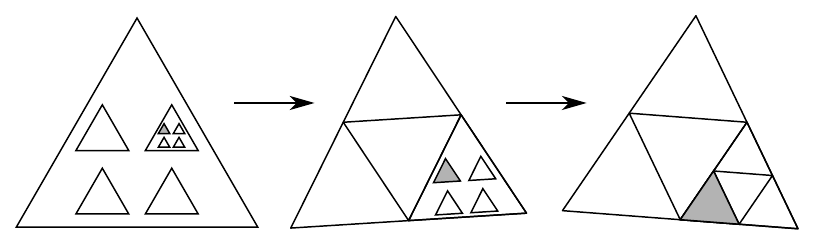}}%
    \put(0.33674546,0.21048155){\color[rgb]{0,0,0}\makebox(0,0)[t]{\lineheight{1.25}\smash{\begin{tabular}[t]{c}$Q_0$\end{tabular}}}}%
    \put(0.67110409,0.21048155){\color[rgb]{0,0,0}\makebox(0,0)[t]{\lineheight{1.25}\smash{\begin{tabular}[t]{c}$Q_1$\end{tabular}}}}%
  \end{picture}%
\endgroup%

\end{center}
\caption{\label{fig:admissible} The maps $P_i$ and $Q_i$ send uncollapsed balls (small triangles) to cells of $X_{i+1}$ by similarities and collapse the region outside the uncollapsed balls to lower-dimensional skeleta. The four mid-sized triangles on the left are the uncollapsed balls of $P_1=Q_0$ and the small triangles are uncollapsed balls of  $P_2=Q_1\circ Q_0$. Note that $P_i$ collapses more and more of $X_0$ to the $1$--skeleton of $X_i$ as $i$ increases.}
\end{figure}

Let
$$P_i=Q_{i-1}\circ \dots \circ Q_0\from X_0\to X_i.$$
If the $Q_i$'s have small displacement in the sense that $d_c(\iota_{i+1}(Q_i(x)), \iota_i(x))\lesssim s^i$, then $\iota_i\circ P_i$ converges uniformly to a locally Hölder map $P\from X_0 \to \H$; indeed, $d_c(\iota_i\circ P_i,P)\lesssim s^i$. Since $\iota_0$ is a bilipschitz homeomorphism from $X_0\to (\H,d_R)$, the map  $g=P\circ \iota_0^{-1}$ will also be locally Hölder.

In fact, this construction results in a map based on nested subdivisions of curves and spheres, like the one outlined above. As $i$ increases, $P_i$ sends larger and larger pieces of $X_0$ to the $1$--skeleton of $X_i$. 
That is, for any $d$, let $M^{(d)}_i=P_i^{-1}(X_i^{(d)})$ be the preimage of the $d$--skeleton of $X_i$.
The cellularity of $Q_i$ implies that $M^{(1)}_{i}\subset M^{(1)}_{i+1}$ for all $i$. 
Let $\delta$ be a $2$--cell of $X_0$ and let $M_{\delta,i}:=\delta\cap M^{(1)}_i$. 
This is the complement of the uncollapsed discs of $P_i$. 
Each uncollapsed disc of $P_{i+1}$ is contained in an uncollapsed disc of $P_{i}$ (see Figure~\ref{fig:admissible}), so $\bigcup_i M^{(1)}_{i}$ contains all of $\delta$ except for a Cantor set.

Let $\gamma:=P(\partial \delta)$; when $s$ is sufficiently small, $\gamma$ will be a $(1-\epsilon)$--Hölder curve in $\H$ with diameter roughly $1$. Such curves are generally not differentiable, so they do not have horizontal velocities, but several authors have noted that they satisfy an integral version of the horizontality condition, see for instance \cite[Lemma~3.1]{LDZ13}.
Let $M_{\delta,i}:=\delta\cap M^{(1)}_i$ be the complement of the uncollapsed discs of $P_i$. 
Then $P_i(M_{\delta,i})\subset X_i^{(1)}$, and there is an $N$ such that $M_{\delta,i}$ is an $N$--holed disc. The outer boundary of $M_{\delta,i}$ is $\partial \delta$, so the images of the inner boundaries of $M_{\delta,i}$ subdivide $\gamma$ into closed curves $\gamma_1,\dots, \gamma_N$, each of diameter roughly $s^{i}$.
As $i$ increases, these subdivisions grow finer and finer, and as in the proof of Theorem~\ref{thm:Hoelder-ext-Carnot-Dehn-step}, the restriction $P|_\delta$ is Hölder.

We see the same picture in the $3$--cells. Let $\Delta$ be a $3$--cell of $X_0$ and let $S:=P(\partial \Delta)$ be a Hölder sphere.
Let $M_{\Delta,i}:=\Delta\cap M^{(2)}_i$; as before, this is $\Delta$ minus some uncollapsed balls. Since $P_i(M_{\Delta,i})$ is $2$--dimensional, $P$ sends the outer boundary of $M_{\Delta,i}$ to $S$ and sends the inner boundaries of $M_{\Delta,i}$ to a subdivision of $S$ into Hölder spheres of diameter roughly $s^{i}$. As $i$ increases, the subdivisions grow finer and finer, so that ultimately $P|_\Delta$ is Hölder.

Though we can write $P$ in terms of subdivisions, our bound on the Hölder exponent takes a different approach. We sketch the bound here; for full details, see Section~\ref{sec:holder-bounds}. The main idea of the Hölder bound is that because $d_c(\iota_i\circ P_i,P)\lesssim s^i$, the Hölder exponent of $P$ is bounded by the growth rate of $\Lip(P_i)$, and we control $\Lip(P_i)$ by placing geometric conditions on the $Q_i$'s.

First, we require that $Q_i$ is Lipschitz on $X_i^{(1)}$ with Lipschitz constant independent of $s$. This ensures that $\Lip(P_i|_{X_0^{(1)}})$ does not grow too quickly. Second, we 
require that there is an $0<a<s$ such that for all $i$, the uncollapsed balls of $Q_i$ are disjoint regular simplices of radius $a s^i$ and that $Q_i$ sends each uncollapsed ball to a cell of $X_{i+1}$ by a similarity with scaling factor $\frac{s}{a}$. 

The Dehn function of the Heisenberg group puts an upper bound on the size of $a$. That is, if  $\delta$ is a $2$--cell of $X_i$, then $Q_i(\partial \delta)$ is a curve of length roughly $s^i$ in $X_{i+1}$, which is to say that $Q_i(\partial \delta)$ is a loop of roughly $s^{-1}$ edges. Filling such a loop with a disc typically requires on the order of $s^{-3}$ triangles, so $Q_i|_\delta$ might have roughly $s^{-3}$ uncollapsed balls. If these all have radius $a s^i$ and are contained in a ball of radius $s^i$, then $a\lesssim s^{\frac{3}{2}}$. In fact, we may choose $Q_i$ so that $a= c s^{\frac{3}{2}}$ for some $c$ independent of $s$.

We claim that there is a $C$ such that 
$$\left(\frac{s}{a}\right)^i\le \Lip(P_i)\le C \left(\frac{s}{a}\right)^i.$$ 
The lower bound follows from considering uncollapsed balls. If $B$ is an uncollapsed ball of $P_i$, then for each $0\le j< i$, $P_j(B)$ is contained in an uncollapsed ball of $Q_j$. Since $Q_j$ scales $P_j(B)$ by $\frac{s}{a}$, the composition $P_i$ scales $B$ by a factor of $\left(\frac{s}{a}\right)^i$. 

The upper bound is more difficult, but the key idea is that for any $x\in X_0$, we may consider the sequence $D_i(x)$ of integers such that $D_i(x)$ is the dimension of the cell of $X_i$ whose interior contains $P_i(x)$. This sequence is non-increasing, and if $D_i(x)>1$ and $P_i(x)$ is not contained in an uncollapsed ball of $Q_i$, then $D_{i+1}(x)<D_i(x)$. Since $D_i(x)\le 3$, this can happen at most three times. That is, for all but at most three values of $i$, either $P_i(x)$ lies in an uncollapsed ball of $Q_i$ or $P_i(x)\in X_i^{(1)}$.
But $Q_i$ is $\frac{s}{a}$--Lipschitz on uncollapsed balls and uniformly Lipschitz on $X_i^{(1)}$, so 
$$\Lip(P_i)\le C \left(\frac{s}{a}\right)^i$$
for some $C$ depending on $\max_i \Lip(Q_i)$. 

We use this bound and the bound $d_c(\iota_i\circ P_i,P)\lesssim s^i$ to show that $P$ is locally $\frac{\log s}{\log a}$--Hölder, see Section~\ref{sec:holder-bounds}. Since $a\approx s^{\frac{3}{2}}$, $\frac{\log s}{\log a}$ approaches $\frac{2}{3}$ as $s\to 0$.

\subsection{Structure of paper} 
In Section~\ref{sec:prelims} we establish basic notation used throughout the text. We furthermore recall the precise definition of the coarse Dehn function $\Ar_{X, \varepsilon}$ and the definition of an admissible map between complexes. Finally, we collect some background on Carnot groups needed for the rest of the paper. In Section~\ref{sec:triang to holder} we deduce Theorem~\ref{thm:Hoelder-ext-Carnot-Dehn-step} from the extension results proved in \cite{LWY16}. 

In Section~\ref{sec:cellulation-and-Q} we construct the triangulations $X_i$ of $\H$ and the maps $Q_i\from X_i\to X_{i+1}$ used in the sketch. In practice, the $X_i$'s are all scalings of one triangulation $X_0$ and the $Q_i$'s are all scalings of one map $X_0\to X_1$. We compose scalings of $Q$ to produce a map $P$ from $X_0$ to $(\H, d_c)$. We give two constructions of $Q$; to prove Theorem~\ref{thm:3d equivariant}, we mainly need bounds on the Lipschitz constant of $Q$ and the volume of its cells, but to prove Theorem~\ref{thm:3d tangent cone}, we need some additional conditions, which are provided by Lemma~\ref{lem:quant-admissible} and Lemma~\ref{lem:normalizing uncollapsed}. These lemmas deal with approximating continuous maps to arbitrary simplicial complexes by admissible maps, and we defer their proofs to Appendix~\ref{app:admissible-maps}.

In Section~\ref{sec:holder-bounds} we show that the map $P$ satisfies the properties of Theorem~\ref{thm:3d equivariant} and prove Theorem~\ref{thm:approximations}. In Section~\ref{sec:Hoelder-R3-Hc} we analyze the tangent cone of $P$ and prove Theorem~\ref{thm:3d tangent cone}. 

\subsection{Acknowledgments}
 The authors would like to thank Assaf Naor for suggesting Theorem~\ref{thm:approximations} and to thank the referee for their hard work reading the original version of the paper and providing critical feedback.

\section{Preliminaries}\label{sec:prelims}

\subsection{Basic definitions and notation}

The Euclidean metric on $\R^d$ will be denoted by $|\cdot|$. For $x\in \R^d$ and $r>0$, we let $B(x,r):=\{x\in \R^d:|x|\leq r\}$ and let $D^d:=B(0,1)$ be the closed unit ball.

Let $(X, d_X)$ be a metric space. The length of a curve $c$ in $X$ is denoted by $\length_{d_X}(c)$ or simply by $\length(c)$. Let $(Z, d_Z)$ be another metric space.  A map $\varphi\colon Z\to X$ is called \emph{$(\lambda, \alpha)$--H\"older} if $$d_X(\varphi(z), \varphi(z'))\leq \lambda\cdot d_Z(z,z')^\alpha$$ for all $z,z'\in Z$.  We say $\varphi$ is \emph{$\alpha$--H\"older} if it is $(\lambda, \alpha)$--H\"older for some $\lambda>0$.

Let $U\subset \R^d$ be open. The (parametrized) volume of a Lipschitz map $\varphi\colon U\to X$ is defined by $$\vol^d(\varphi) = \vol^d(\varphi; X) = \int_X\#\{z\in U: \varphi(z)= x\}\,d\mathcal{H}^d(x),$$ where $\mathcal{H}^d$ denotes the $d$--dimensional Hausdorff measure on $X$. If $X$ is a Riemannian manifold or a CW complex with piecewise Riemannian metric then, by the area formula, $\vol^d(\varphi)$ agrees with the volume defined by integrating the jacobian of the derivative of $\varphi$.

\subsection{Admissible maps} \label{sec:complexes-and-maps}
Let $X$ be a simplicial complex.  For $d\geq 0$ we denote by $\cF^d(X)$ the set of closed $d$--simplices in $X$ and by $X^{(d)}$ the $d$--skeleton of $X$. Let $\cF(X)$ be the set of cells in $X$ of every dimension. 

We say that a continuous map $f\from D^n\to X^{(n)}$  is an \emph{admissible map} if there is a collection of closed subsets $B_1,\dots, B_k\subset D^n$, called \emph{uncollapsed balls}, such that the interiors of the $B_i$ are pairwise disjoint, $f$ sends each $B_i$ homeomorphically to an $n$--cell of $X$, and $f(D^n\setminus \bigcup B_i)\subset X^{(n-1)}$.  We call $k$ the \emph{admissible volume} of $f$, denoted $\avol(f;X)$ or simply $\avol(f)$ when it is clear what the target complex is. This differs slightly from the definition given in \cite{BBFS}, where the target space is a CW complex and the uncollapsed balls are open balls.

Given simplicial complexes $X$ and $Y$, we call a continuous map $f\from X\to Y$ \emph{admissible} if $f$ is cellular (i.e., $f(X^{(i)})\subset Y^{(i)}$ for all $i$) and for every $d$--cell $\sigma\in \cF^d(X)$ with $d>0$, the map $f|_\sigma$ is admissible in the sense above. Note that if $X$, $Y$, and $Z$ are simplicial complexes and $f\from X\to Y$ and $g\from Y\to Z$ are admissible, then $g\circ f$ is admissible.  

\subsection{The coarse Dehn function}

The coarse Dehn function which we introduce here is a slight variant of Gromov's definition of mesh function given in \cite{GroAII}, see also \cite[III.H.2.1]{BrH99} or \cite{Dru02}. A triangulation of the closed unit disc $D^2$ is a homeomorphism from $D^2$ to a combinatorial $2$--complex $\tau$ in which every $2$--cell is a triangle. We endow $D^2$ with the cell-structure of $\tau$.
Let $X$ be a geodesic metric space and $c\colon S^1\to X$ a Lipschitz curve. Let $\epsilon>0$. An $\epsilon$-filling of $c$ is a pair $(P, \tau)$ consisting of a triangulation $\tau$ of $D^2$ and a continuous map $P\colon \tau^{(1)}\to X$ such that $P|_{S^1}=c$ and such that $\length(P|_{\partial F})\leq \epsilon$ for every triangle $F$ in $\tau$. The $\epsilon$-area of $c$ is $$\Ar_\epsilon(c):= \min\left\{|\tau| \mid \text{$(P, \tau)$ is an $\epsilon$-filling of $c$}\right\}.$$
Here, $|\tau|$ denotes the number of triangles in $\tau$. If no $\epsilon$-filling exists then we set $\Ar_\epsilon(c):= \infty$. The \emph{$\epsilon$-coarse Dehn function} of $X$ is defined by $$\Ar_{X, \epsilon}(r):= \sup\left\{\Ar_\epsilon(c)\mid \text{$c\colon S^1\to X$ Lipschitz, $\length(c)\leq r$}\right\}$$ for $r>0$.

It is not difficult to show that the asymptotic growth of $\Ar_{X, \epsilon}(r)$ is a quasi-isometry invariant. Moreover, under mild conditions on the underlying space the function $\Ar_{X,\epsilon}(r)$ has the same growth as the so-called Lipschitz Dehn function $\delta^{\rm Lip}_X(r)$ (see \cite{LWY16}). We will however not need this. Recall here that the Lipschitz Dehn function $\delta^{\rm Lip}_X(r)$ measures how much (parametrized) area is needed to fill a curve of length $r$ by a Lipschitz disc in $X$.

\subsection{Carnot groups and Carnot-Carath\'eodory distance}\label{sec:Carnot}

A connected, simply-connected nilpotent Lie group $G$ is called a \emph{Carnot group} if its Lie algebra $\mathfrak{g}$ admits a stratification into subspaces $$\mathfrak{g} = V_1\oplus\dots\oplus V_k$$ such that $[V_1, V_i] = V_{i+1}$ for all $i=1,\dots, k-1$ and $[V_1,V_k] = 0$. Here, $[V_1,V_i]$ is the subspace spanned by the elements $[v, v']$ with $v\in V_1$ and $v'\in V_i$. The number $k$ is called the {\it step} or {\it nilpotency class} of $G$.

As $G$ is a simply-connected nilpotent Lie group, the exponential map $\exp\colon \mathfrak{g}\to G$ is a diffeomorphism.  Since it is Carnot, $G$ comes with a family of scaling homomorphisms $\delta_r\colon G\to G$ for $r\geq 0$. They are given by $$\delta_r(\exp(v)) = \exp\left(\sum_{i=1}^k r^i v_i\right)$$ for $v= v_1+\dots+v_k$ with $v_i\in V_i$. The derivative $D_x\delta_r$ of $\delta_r$ at $x\in G$ is given by 
\begin{equation}\label{eq:derivative-dil-hom}
 D_x\delta_r(D_0L_x(v)) = r^i\cdot D_0L_{\delta_r(x)}(v)
\end{equation}
 for all $v\in V_i$, where $L_x$ denotes left-translation by $x$ and $0$ is the identity element of $G$.

The \emph{horizontal bundle} $TH$ is the subbundle of $TG$ obtained by left-translation of the subspace $V_1$. Let $g_0$ be a left-invariant Riemannian metric on $G$ such that at $0$ the subspaces $V_i$ are pairwise orthogonal with respect to $g_0$. Let $d_R$ be the distance coming from $g_0$. A curve $c\colon [0,1]\to G$ which is absolutely continuous with respect to $d_R$ is called \emph{horizontal} if $c'(t)\in T_{c(t)}H$ for almost every $t\in[0,1]$. The \emph{Carnot-Carath\'eodory distance} (or \emph{Carnot metric} for short) associated with $d_R$ is defined by $$d_c(x,y):= \inf\{\length_R(c) \mid \text{ $c\colon[0,1]\to G$ horizontal curve from $x$ to $y$}\}.$$
Here, $\length_R(c)$ denotes the length of $c$ with respect to the metric $d_R$. It can be shown that the Carnot-Carath\'eodory distance is always finite and thus defines a metric. It is moreover left-invariant and $1$-homogeneous with respect to the scaling automorphisms; that is, $$d_c(\delta_r(x), \delta_r(x')) = r d_c(x,x')$$ for all $x,x'\in G$ and all $r>0$. The topologies induced on $G$ by $d_R$ and $d_c$ agree, however it is well-known that these metrics are not even locally bilipschitz equivalent (except when $G$ is abelian). Throughout this article, when we talk about a Carnot metric on $G$ we always mean one which is associated with a distance coming from a left-invariant Riemannian metric such that at $0$ the subspaces $V_i$ are orthogonal.

We will need the following simple facts.

\bl\label{lem:rel-d0-dc-Carnot}
 Let $G$ be a Carnot group of step $k$. Let $d_R$ be the distance coming from a left-invariant Riemannian metric on $G$ and let $d_c$ be the associated Carnot metric. Then there exists $L\geq 1$ such that 
 \begin{enumerate}
  \item[(i)] $d_R(x,x')\leq d_c(x,x') \leq L\cdot d_R(x,x') + L$ for all $x,x'\in G$. 
  \item[(ii)] for every $x\in G$ the curve $\gamma(r):= \delta_r(x)$ satisfies $$\length_R(\gamma|_{[s,t]}) \leq L\cdot (d_c(x,0) + 1)^k\cdot |t-s|$$ for all $0\leq s\leq t\leq 1$.
 \end{enumerate}
\el

We provide a short proof for the convenience of the reader, compare with \cite{Pit95}.

\begin{proof}
 The first inequality in (i) is clear. To prove the second inequality in (i), notice that 
 \begin{equation*}
      L:= \sup\{d_c(y,z): d_R(y,z)\leq 1\}= \sup\{d_c(0,x): d_R(0,x)\leq 1\}<\infty.
 \end{equation*}
 Let $x,x'\in G$ be distinct and set $l:= d_R(x,x')$. Let $\gamma\from [0,l]\to (G,d_R)$ be a geodesic from $x$ to $x'$, parametrized by arc-length. Let $m$ be the largest integer smaller than $l$. Then $$d_c(x,x')\leq \sum_{i=1}^m d_c(\gamma(i-1), \gamma(i)) + d_c(\gamma(m), \gamma(l))\leq L\cdot d_R(x,x') + L,$$ which proves the second inequality in (i).
 
 We now prove statement (ii). For $x\in G$ define the curve $\gamma_x\from [0,1]\to G$ by $\gamma_x(r):= \delta_r(x)$. Notice that $$L:= \sup\{\|\gamma_x'(r)\|: d_R(0,x)\leq 1, r\in[0,1]\}<\infty.$$ Let $x\in G$ and $0\leq s<t\leq 1$. Suppose that $d_R(0,x)\leq 1$. Then $$\length_R(\gamma_x|_{[s,t]})= \int_s^t\|\gamma_x'(r)\|\,dr \leq L\cdot |t-s|\leq L\cdot(d_c(x,0)+1)^k\cdot |t-s|.$$ 
 
 Now suppose that $d_R(0,x)>1$. Let $\beta$ be a geodesic with respect to $d_R$ from $0$ to $x$. By \eqref{eq:derivative-dil-hom}, if $\xi\in (0,1)$, then $$\length_R(\delta_\xi\circ\beta) \le \xi \length_R(\beta)= \xi d_R(0,x).$$
 Let $\xi=(d_R(0,x))^{-1}$, so that 
 $d_R(0,\delta_\xi(x)) \le \length_R(\delta_\xi\circ\beta) \le 1.$
 
 By \eqref{eq:derivative-dil-hom}, $\delta_{\xi^{-1}}$ distorts length by at most a factor of $\xi^{-k}$, so
 $$\length_R(\gamma_x|_{[s,t]})\leq \xi^{-k}\cdot \length_R(\delta_\xi\circ \gamma_x|_{[s,t]})=\xi^{-k}\cdot \length_R(\gamma_{\delta_\xi(x)}|_{[s,t]})\leq L\cdot \xi^{-k}\cdot |t-s|$$
 and thus $$\length_R(\gamma_x|_{[s,t]})\leq L\cdot (d_R(0,x))^k\cdot |t-s|\leq L\cdot (d_c(0,x))^k\cdot |t-s|.$$ This proves statement (ii).
\end{proof}

Finally, we consider the first Heisenberg group which is the Carnot group of step $2$ given by $\H:= \R^3$ with the multiplication 
\begin{equation}\label{eq:heismult}
(x,y,z) \cdot (x',y',z') = (x+x', y+y', z+z' + xy').
\end{equation}
A basis of left-invariant vector fields on $\H$ is given by $$X = \frac{\partial}{\partial x}, \quad Y= \frac{\partial }{\partial y} + x\frac{\partial}{\partial z}, \quad Z = \frac{\partial}{\partial z},$$ and the Lie algebra $\mathfrak{h}$ of $\H$ has the stratification $\mathfrak{h} = \operatorname{span}\{X, Y\}\oplus \operatorname{span}\{Z\}$. We denote by $\H_\Z:=\Z^3$ the integer lattice in $\H$.

Let $g_0$ be a left-invariant Riemannian metric on $\H$ such that at $0$ the subspaces $V_1$ and $V_2$ are orthogonal. It follows from \eqref{eq:derivative-dil-hom} that for $r\geq 1$ the scaling automorphism $\delta_r$ on $\H$ distorts $2$--dimensional area in $(\H, g_0)$ at most by a factor $r^3$ and $3$--dimensional volume by a factor of exactly $r^4$.

\section{From triangulations to H\"older extensions}
\label{sec:triang to holder}

In this section we prove Theorem~\ref{thm:Hoelder-ext-Carnot-Dehn-step}. The main ingredient is the next theorem, which follows from the results proved in \cite{LWY16}.

\bt\label{thm:subdiv-implies-Hoelder-ext}
 Let $X$ be a complete, geodesic metric space. Suppose there exist $K\geq 1$ and $\beta\geq 2$ such that for all $r>0$ and $n\in\N$ we have $$\Ar_{X, \frac{r}{n}}(r) \leq K\cdot n^\beta.$$ Then for every $\alpha<\frac{2}{\beta}$ the pair $(\R^2, X)$ has the $\alpha$--H\"older extension property.
\et

\begin{proof}
  Let $K\geq 1$ and $\beta\geq 2$ be as in the statement of the theorem and let $\alpha<\frac{2}{\beta}$. Choose $n$ so large that $$\eta:= \frac{\log n}{\log (2\sqrt{Kn^\beta})} = \frac{\log n}{\log 2 + \frac{1}{2}\cdot \log K + \frac{\beta}{2}\cdot \log n} > \alpha.$$ It follows from Proposition 7.4 in \cite{LWY16} that the space $X$ is $\eta$--H\"older $1$--connected and, in particular, also $\alpha$--H\"older $1$--connected. Thus, there exists $C\geq 1$ such that every $\lambda$-Lipschitz curve $c\colon S^1\to X$ admits a $(C\lambda, \alpha)$--H\"older extension to $D^2$. Now, Theorem 6.4 in \cite{LWY16} implies that the pair $(\R^2, X)$ has the $\alpha$--H\"older extension property, which completes the proof.
\end{proof}

We now apply the theorem to Carnot groups:

\bc\label{cor:Carnot-Ar-Hoelder-ext}
 Let $G$ be a Carnot group and let $d_c$ be a Carnot metric on $G$. Let $K\geq 1$ and $\beta\geq 2$ be such that for all $r\geq 1$, $$\Ar_{(G, d_c), 1}(r)\leq K\cdot r^\beta.$$ Then the pair $(\R^2, (G, d_c))$ has the $\alpha$--H\"older extension property for all $\alpha<\frac{2}{\beta}$.
\ec

\begin{proof}
 Let $n\in \N$ and $r>0$. Since scaling automorphisms are $1$--homo\-geneous with respect to $d_c$ we have $$\Ar_{(G, d_c), \frac{r}{n}}(r) = \Ar_{(G, d_c), 1}(n)\leq K\cdot n^\beta$$ and hence the corollary follows from Theorem~\ref{thm:subdiv-implies-Hoelder-ext}.
\end{proof}

Now, Theorem~\ref{thm:Hoelder-ext-Carnot-Dehn-step} follows from the corollary above together with the next lemma.

\bl\label{lem:bound-coarse-Dehn-Carnot-step}
 Let $G$ be a Carnot group of step $k$ and let $d_c$ be a Carnot metric on $G$. Then there exists $K\geq 1$ such that $$\Ar_{(G, d_c), 1}(r)\leq K\cdot r^{k+1}$$ for all $r\geq 1$. 
\el

The lemma could easily be deduced from the upper bound on the growth of the Lipschitz Dehn function proved in \cite{Pit95} and the fact that the coarse Dehn function has the same growth as the Lipschitz Dehn function (see e.g.~\cite{LWY16}). We prefer to give a direct proof here.

\begin{proof}
 Let $d_R$ be a left-invariant Riemannian metric on $G$ such that $d_c$ is the Carnot metric associated  with $d_R$. Let $L\geq 1$ be as in Lemma~\ref{lem:rel-d0-dc-Carnot}. 
 
 Using the scaling homomorphisms it suffices to show that there exists $K$ such that 
 \begin{equation}\label{eq:Carnot-bound-Ar}
  \Ar_{(G, d_c), 6L}(r)\leq K\cdot r^{k+1}
 \end{equation}
  for all $r\geq 6L$. In order to prove \eqref{eq:Carnot-bound-Ar} let $c\colon S^1\to (G, d_c)$ be a Lipschitz curve of length $r\geq 6L$. We may assume that $c$ has constant speed and that $c$ passes through $0\in G$. 
  
  We construct a $6L$-filling $(P, \tau)$ of $c$ as follows. Let $M\in\N$ be the smallest integer larger than $\frac{r}{L}$ and set $v_l:= e^{\frac{2\pi l i}{M}}$ for $l=0,\dots, M-1$. Furthermore, let $m$ be the smallest integer larger than $L\cdot (r+1)^k$. We define a triangulation $\tau$ of $D^2$ as follows. The set of vertices of $\tau$ is the subset of $D^2$ given by the points $v_{0,0}=0$ and $v_{j,l}:= \frac{j}{m}\cdot v_l$ for $j=1,\dots, m$ and $l=0,\dots, M-1$. For $l=0,\dots, M-1$ we add edges between the points $v_{j, l}$ and $v_{j+1,l}$ for $j=0,\dots, m-1$, where we have set $v_{0, l} = v_{0,0}$, between $v_{j,l}$ and $v_{j, l+1}$ for $j=1,\dots, m$, where $v_{j, M} = v_{j, 0}$,  and between $v_{j,l}$ and $v_{j+1, l+1}$ for $j=1,\dots, m-1$. This gives a triangulation $\tau$ of $D^2$ consisting of less than $2mM$ triangles and thus $|\tau| \leq K r^{k+1}$, where $K$ only depends on $L$ and $k$. We now define a map $P\colon \tau^{(1)}\to G$ by setting $P|_{S^1}:= c$ and by defining $P(v_{j,l}):= \delta_{\frac{j}{m}}(c(v_l))$. We extend $P$ to $\tau^{(1)}$ by mapping the edges in $\tau^{(1)}\setminus S^1$ to geodesics with respect to the $d_c$-distance. It remains to show that $\length_c(P|_{\partial F})\leq 6L$ for every triangle $F$ in $\tau$, where $\length_c$ denotes the length with respect to the Carnot metric $d_c$. For this we use the estimates in Lemma~\ref{lem:rel-d0-dc-Carnot}. Firstly, we have $$d_c(P(v_{j,l}), P(v_{j, l+1})) = \frac{j}{m}\cdot d_c(c(v_l), c(v_{l+1})) \leq \frac{j}{m}\cdot \length_c(c|_{[v_l, v_{l+1}]})\leq L.$$ Secondly, writing $\gamma_l(t):= \delta_t(c(v_l))$ we obtain that $$d_c(P(v_{j, l}), P(v_{j+1, l})) \leq L\cdot \length_R(\gamma_l|_{\left[\frac{j}{m}, \frac{j+1}{m}\right]}) + L\leq L^2\cdot (r+1)^k\cdot \frac{1}{m} + L \leq 2L.$$ The two inequalities together finally yield $$d_c(P(v_{j, l}), P(v_{j+1, l+1})) \leq 3L,$$ from which it follows that $\length_c(P|_{\partial F})\leq 6L$ for every triangle $F$ in $\tau$. This proves \eqref{eq:Carnot-bound-Ar}.\end{proof}

\section{Constructing H\"older maps from admissible maps}\label{sec:cellulation-and-Q}

While one can construct Hölder maps from discs to Carnot groups by constructing a sequence of nested subdivisions, it is difficult to generalize this construction to higher-dimensional domains. Instead, we will construct Hölder maps using admissible maps.

As in the introduction, we will construct a sequence of simplicial complexes $X_0, X_1, \dots$, equipped with path metrics, a sequence of maps $\iota_i\from X_i\to \H$, and a sequence of admissible maps $Q_i\from X_i\to X_{i+1}$. We then define $P_i\from X_0\to X_i$,
$$P_i:=Q_{i-1}\circ \dots \circ Q_0,$$
and let $P\from X_0 \to \H$, $P:=\lim_i \iota_i\circ P_i$. The map $P$ is the desired Hölder map.
In this section, we will give the details of the construction by defining $X_i$, $\iota_i$, and $Q_i$, and in the next section, we will show that $P$ is Hölder.

We will define the $X_i$'s to be scalings of a single triangulation $X$ of $\H$.
Recall that $\Hz\subset \H$ is the lattice of integer points.
Let $X$ be an $\Hz$--equivariant piecewise-linear (PL) triangulation of $\H$. That is, $X$ is a tuple $X=(Y,\iota, \zeta)$ of a simplicial complex $Y$, a PL homeomorphism $\iota\from Y\to \H$, and a left $\Hz$--action $\zeta[g]\from Y \to Y$, $g\in \Hz$, such that $\iota(\zeta[g](x))=g\iota(x)$ for all $g\in \Hz$, $x\in Y$. The existence of such a $\iota$ relies on the fact that by \eqref{eq:heismult}, if $\alpha\from \R^n\to \H$ is affine, then for any $g\in \H$, the translate $v\mapsto g\alpha(v)$ is also affine.
This induces a subdivision of $\H$ into simplices $\big(\iota(\sigma)\big)_{\sigma\in \cF(Y)}$. Since $\iota$ is equivariant, the action of $\Hz$ on $\H$ permutes the simplices in this subdivision.
We write $x\in X$ to denote a point in the underlying simplicial complex $Y$ of $X$ and define $gx=\zeta[g](x)$ for $g\in \Hz$, $x\in X$.
Let $d_\Delta$ be the standard metric on $X$, i.e., the path metric such that every simplex of $X$ is isometric to the unit regular simplex. Since $\iota$ is PL, it is a bilipschitz equivalence from $(X,d_\Delta)$ to $(\H,d_R)$, where $d_R$ is the Riemannian metric on $\H$.

Let $n>1$ be a large integer to be chosen later, and let $s=n^{-1}$. For each $i$ and for each $i\ge 0$, let $X_i$ be $X$ scaled by $s^i$. That is, $X_i=(Y, \iota_i, \zeta_i)$, where $\iota_i=\delta_{n^{-i}}\circ \iota$ and $\zeta_{i}[g](x)=\zeta[\delta_{n^i}(g)](x)=\delta_{n^i}(g)x$ for all $g\in \Hz$, $x\in X$.
Then $X_i$ is a triangulation of $\H$, and for each $\sigma\in \cF(X_i)$ with $\dim \sigma>0$, the simplex $\iota_i(\sigma)$  has $d_c$--diameter roughly $s^i$. Furthermore, $\iota_i$ is $\Hz$--equivariant, i.e., for all $x\in X_i$,
$$\iota_i(\zeta_{i}[g](x)) = \delta_{n^{-i}}\big(\iota\big(\zeta[\delta_{n^{i}}(g)](x)\big)\big) = \delta_{n^{-i}}\big(\delta_{n^{i}}(g)\iota(x)\big) = g \delta_{n^{-i}}(\iota(x)) = g\iota_i(x).$$
As before, for $x\in X_i$, we let $gx$ denote $\zeta_{i}[g](x)$.

We equip $X_i$ with the metric $d_i=s^id_\Delta$. Since $\iota$ is a bilipschitz equivalence from $d_\Delta$ to $d_c$ and $d_c$ is quasi-isometric to $d_R$ (Lemma~\ref{lem:rel-d0-dc-Carnot}), $\iota$ is a quasi-isometry from $(X,d_\Delta)$ to $(\H,d_c)$. That is, there is a $k>1$ such that for any $x,y\in X_0$,
\begin{equation}\label{eq:X-qi}
k^{-1} d_\Delta(x,y)-k\le d_c(\iota(x),\iota(y)) \le k d_\Delta(x,y)+k.
\end{equation}
By the scale-invariance of $d_c$, $d_c(\iota_i(x),\iota_i(y))=s^{i} d_c(\iota(x),\iota(y))$, so multiplying \eqref{eq:X-qi} by $s^i$ gives
\begin{equation}\label{eq:Xi-qi}
k^{-1} d_i(x,y)-ks^i\le d_c(\iota_i(x),\iota_i(y)) \le k d_i(x,y)+ks^i.
\end{equation}
That is, $(X_i,d_i)$ approximates $(\H,d_c)$ more and more closely as $i$ increases.

We will construct $P\from X_0\to \H$ from a sequence of $\Hz$--equivariant admissible maps $Q_i\from X_i\to X_{i+1}$. Since all of the $X_i$'s have the same underlying simplicial complex, we can construct the $Q_i$'s as a scaling of a map $Q\from X\to X$ which is equivariant as a map from $X_0$ to $X_1$, that is, $Q(\zeta_0[g](x)) = \zeta_1[g](Q(x))$ for all $g\in \Hz$, $x\in X$. Since $\zeta_{i}[g](x)=\zeta[\delta_{n^i}(g)](x)$ for all $g\in \Hz$, $x\in X$, we have $Q(\zeta[g](x))=\zeta[\delta_n(g)](Q(x))$ and
$$Q(\zeta_{i}[g](x)) = Q(\zeta[\delta_{n^i}(g)](x)) = \zeta[\delta_{n^{i+1}}(g)](Q(x)) = \zeta_{i+1}[g](Q(x))$$
for any $i\ge 0$. That is, if $Q$ is $\Hz$--equivariant as a map from $X_0$ to $X_1$, then it is $\Hz$--equivariant as a map from $X_i$ to $X_{i+1}$ for any $i\ge 0$.

We construct $Q$ so that it satisfies the following properties.
A \emph{similarity} from one simplex to another is a map that scales the metric by a constant factor, i.e., a rotation or reflection followed by a scaling.
\begin{lemma}\label{lem:Qi-conds}
  There are $c,n_0>1$ such that for every $n>n_0$ there is an admissible, surjective, degree--1, $\Hz$--equivariant Lipschitz map $Q\from X_0\to X_1$ that has the following properties. Let $Q_i\from (X_i,d_i)\to (X_{i+1},d_{i+1})$ denote $Q$ viewed as a map from $X_i$ to $X_{i+1}$.
  % For every $d=2,3$ and every $\sigma\in \cF^d(X)$, let $q_\sigma\from D^d(c_ds^i) \to X_1$ be the map $q_\sigma=Q_i\circ r_\sigma$. Then:
  \begin{enumerate}
  \item For all $i$, $\Lip(Q_i)=\Lip(Q)$ and $\Lip(Q_i|_{X_i^{(1)}})< c$.
  \item $Q_i$ has small displacement. That is, for every $x\in X_i$, 
  $$d_c(\iota_i(x),\iota_{i+1}(Q_i(x)))< cs^i.$$
  \item Let $a=c^{-1} n^{-\frac{3}{2}}$; note that $0<a<s<1$.
    For $\sigma\in \cF^d(X_i)$ with $d=2,3$, every uncollapsed ball $B$ of $Q_i|_\sigma$ is a regular simplex with edges of length $a s^i$, i.e., $B$ is a copy of $\sigma$ scaled by $a$.
  \item 
    For any uncollapsed ball $B$, the restriction $Q_i|_B$ is a similarity from $B$ to a simplex of $X_{i+1}$, with scale factor $\frac{s}{a}$.
  \item
    There is an $r>0$ independent of $n$ such that for any $\sigma\in \cF^3(X_0)$, there is a $q\in \H$ such that
    $$d_c(q,\iota_{1}(Q(X_0\setminus \sigma))) > r.$$
  \end{enumerate}
\end{lemma}
Before we prove the lemma, we will sketch how to use the lemma  to construct Hölder maps. Indeed, the map in Theorem~\ref{thm:3d equivariant} can be constructed from a map $Q_i$ that satisfies conditions (1), (3), (4), and the condition that there is some $D(n)>0$ such that
\begin{equation}\label{eq:weak-cond-2}
d_c(\iota_i(x),\iota_{i+1}(Q_i(x)))< D(n) s^i \qquad \text{ for all } x\in X.
\end{equation}
This is a version of (2) in which the bound on displacement is allowed to depend on $n$.

Given $Q$ and $Q_i$ satisfying these conditions, we let $P_i\from X_0\to X_i$, $P_i=Q_{i-1}\circ \dots \circ Q_0$. We claim that $(\iota_i\circ P_i)_i$ converges uniformly as $i\to \infty$. 
By equation \eqref{eq:weak-cond-2}, for any $x\in X_0$, we have
$$d_c\big(\iota_i(P_i(x)),\iota_{i+1}(P_{i+1}(x))\big) = d_c\big(\iota_i(P_i(x)),\iota_{i+1}(Q_i(P_{i}(x)))\big) < D(n)s^i,$$
so $(\iota_i\circ P_i)_i$ is uniformly Cauchy. Let $P\from X_0 \to \H$, $P=\lim_i \iota_i\circ P_i$. Since $s\le \frac{1}{2}$,
\begin{equation}\label{eq:unif Cauchy D}
  d_c\big(\iota_i(P_i(x)),P(x)\big) < \frac{D(n) s^i}{1-s} \le 2D(n)s^i.
\end{equation}

Combining \eqref{eq:Xi-qi} and \eqref{eq:unif Cauchy D} with a bound on the growth of the Lipschitz constants of the $P_i$'s leads to Hölder bounds on $P$. For example, condition (1) of Lemma~\ref{lem:Qi-conds} implies the following bound.
\begin{lemma}\label{lem:1+e-holder}
  Let $Q_i$, $P_i$, and $P$ be as above and let $\epsilon>0$. If $n$ is sufficiently large, then $P|_{X_0^{(1)}}$ is locally $(1-\epsilon)$--Hölder.
\end{lemma}
\begin{proof}
  Let $x,y\in X_0^{(1)}$ such that $0<d_0(x,y)<1$.
  Let $\rho=\frac{s}{c}<1$ and let $i\ge 0$ be the integer such that $\rho^{i+1}\le d_0(x,y) < \rho^i$. By condition (1) of Lemma~\ref{lem:Qi-conds} and the fact that $Q_i$ is cellular, we have $\Lip(P_i|_{X_0^{(1)}})\le c^i$. Let $k$ be as in \eqref{eq:Xi-qi}. Then by \eqref{eq:unif Cauchy D},
  \begin{multline*}
    d_c(P(x),P(y)) \le 4D(n)s^i + d_c(\iota_i(P_i(x)), \iota_i(P_i(y))) \\
    \le 4D(n)s^i + k c^i \rho^i + k s^i = \frac{4D(n)+2k}{s}s^{i+1} \le K' d_0(x,y)^{\frac{\log s}{\log \rho}},
  \end{multline*}
  where $K'=K'(n)=\frac{4D(n)+2k}{s}$. As $n\to \infty$,
  $$\lim_{n\to \infty} \frac{\log s}{\log \rho} = \lim_{n\to \infty} \frac{- \log n}{- \log n - \log c} = 1.$$
  When $n$ is large enough that $\frac{\log s}{\log \rho}>1-\epsilon$,
  $$d_c(P(x),P(y))\le K' d_0(x,y)^{\frac{\log s}{\log \rho}}\le K' d_0(x,y)^{1-\epsilon},$$
  as desired.
\end{proof}

That is, a bound of the form $\Lip(P_i)\lesssim K(n) L^i$ implies that $P$ is locally $\frac{\log s}{\log \frac{s}{L}}$--Hölder. To prove Theorem~\ref{thm:3d equivariant}, it suffices to show that $\Lip(P_i)\lesssim K(n)\cdot (\frac{s}{a})^i$ for some $K(n)$, and we will use conditions (3) and (4) to prove this bound in  Section~\ref{sec:holder-bounds}.

When $n$ is sufficiently large, condition (5) lets us show that for any $\sigma\in \cF^3(X_0)$, there is a $y\in \sigma$ such that 
$$d_c(P(y),P(X_0\setminus \sigma)) > r - 2cs > \frac{r}{2};$$
we will use this in Section~\ref{sec:Hoelder-R3-Hc} to construct a proper tangent cone of $P$ and prove Theorem~\ref{thm:3d tangent cone}.

Since Theorem~\ref{thm:3d equivariant} relies only on the existence of a map satisfying equation \eqref{eq:weak-cond-2} and conditions (1), (3), and (4), we will sketch the construction of such a map for readers who are primarily interested in Theorem~\ref{thm:3d equivariant}.
For each vertex $v\in \cF^0(X_0)$, choose $Q(v)$ to be a vertex of $X_1$ that minimizes the distance $d_c(\iota_0(v),\iota_1(Q(v)))$. By \eqref{eq:Xi-qi}, for each edge $e=[v,w]$, we have $d_1(Q(v),Q(w))\lesssim 1$. Define $Q$ on $e$ to be a shortest path in $X_1^{(1)}$ from $Q(v)$ to $Q(w)$, parametrized with constant speed. There is a $c_0>0$ independent of $n$ such that this path has at most $c_0n$ edges, so $\Lip(Q|_{X_0^{(1)}})\le c_0$, i.e., $Q$ satisfies condition (1).

For every triangle $\sigma\in \cF^2(X_0)$, the image $Q(\partial \sigma)$ is a closed curve in $X_1^{(1)}$ with at most $3c_0n$ edges. Since the Heisenberg group has a cubic Dehn function, results of \cite{BBFS} imply that there is a $c_1>0$ such that we can extend $Q$ to an admissible map on $\sigma$ with $\avol(Q|_\sigma)\le c_1 n^3$. We can take the uncollapsed discs of $Q|_\sigma$ to be a collection of disjoint balls. There is a homeomorphism from $\sigma$ to itself that sends these disjoint balls to a collection of disjoint regular simplices of edge length roughly 
$$\left(\frac{1}{c_1 n^3}\right)^{\frac{1}{2}}\approx n^{-\frac{3}{2}}.$$
We compose $Q$ with this homeomorphism so that it satisfies condition (3), and by applying a further homotopy, we can ensure that it satisfies condition (4), see Lemma~\ref{lem:normalizing uncollapsed} and Appendix~\ref{app:admissible-maps}.

Now, for every tetrahedron $\Delta\in \cF^3(X_0)$, the image $Q(\partial \Delta)$ is a $2$--sphere in $X_1$ consisting of at most $4c_1n^3$ triangles. The filling volume function of the Heisenberg group grows like $\FV^3(t)\approx t^{\frac{4}{3}}$, so, using results of \cite{BBFS} again, we can extend $Q$ to an admissible map on $\Delta$ with disjoint uncollapsed balls and  $\avol(Q|_\Delta)\le \FV^3(4c_1n^3)\le c_2 n^4$.
There is a homeomorphism that sends the uncollapsed balls of $Q|_\Delta$ to disjoint simplices, so we can adjust $Q$ to satisfy conditions (3) and (4). 

All of these constructions can be done equivariantly, so $Q$ can be chosen to be $\Hz$--equivariant.
Since $\Hz$ acts cocompactly on $\H$, the map $\iota_1\circ Q \circ \iota_0^{-1}$ has bounded displacement; say that there is some $D(n)$ such that
$$d_c(\iota_0(x),\iota_1(Q(x)))\le D(n)$$
for all $x\in X$. Then
$$d_c(\iota_i(x),\iota_{i+1}(Q_i(x)))=d_c(\delta_{s^i}(\iota_0(x)),\delta_{s^i}(\iota_{1}(Q(x)))) \le s^i D(n),$$
so equation \eqref{eq:weak-cond-2} holds. Then $Q$ and the $Q_i$ satisfy the desired conditions, so $Q$ can be used in the proof of Theorem~\ref{thm:3d equivariant}. Readers who are primarily interested in Theorem~\ref{thm:3d equivariant} may turn to Section~\ref{sec:holder-bounds} for the proof.

\medskip 

To achieve conditions (2) and (5) as well, we need stronger control over the geometry of $Q$. We thus construct $Q$ by first constructing a map $\phi\from X_0\to \H$ which is close to $\iota$, then approximating $\iota_1^{-1}\circ\phi$ by an admissible map $A\from X_0\to X_1$, and finally adjusting $A$ to satisfy conditions (3) and (4). Let $\kappa > 0$ be small enough that for every $3$--simplex $\Delta\in \cF^3(X_0)$, the image $\iota_0(\Delta)$ contains a $d_c$--ball of radius $\kappa$.
Any curve in $\H$ can be approximated arbitrarily closely by a horizontal curve, so there is an $\Hz$--equivariant PL map $\phi\from X_0 \to \H$ such that $\phi$ is horizontal on every edge of $X_0$ and 
\begin{equation}\label{eq:iota-kappa}
d_c(\phi(x),\iota_0(x))< \frac{\kappa}{2}
\end{equation}
for all $x\in X_0$.
Then $\iota_1^{-1}\circ \phi$ is an $\Hz$--equivariant PL map from $X_0$ to $X_1$.
By the following lemma, which is proved in Appendix~\ref{app:admissible-maps}, we can approximate $\iota_1^{-1}\circ \phi$ by an admissible map $A\from X_0\to X_1$. 
\begin{lemma}\label{lem:quant-admissible}
    Let $Y$, $Z$ be finite dimensional simplicial complexes equipped with the standard metric, and let $\psi\from Y\to Z$ be a Lipschitz map. Then there is an admissible Lipschitz map $\overline{\psi}\from Y\to Z$ such that:
    \begin{enumerate}
    \item 
        For a subset $S\subset Z$, let $\supp S$ be the smallest subcomplex of $Z$ containing $S$. For any simplex $\sigma\in \cF(Y)$, we have $\overline{\psi}(\sigma) \subset \supp \psi(\sigma)$.
    \item 
        For any edge $e\in \cF^1(Y)$, $\overline{\psi}(e)$ is an edge path parametrized with constant speed.
        %, and 
        %$$\ell(\overline{\psi}|_e) \le b \ell(\psi|_e) + b.$$
    \item 
        For any simplex $\sigma\in \cF(Y)$,
        \begin{equation}\label{eq:quant-avol-bound}
        \avol(\overline{\psi}|_\sigma) \le b\cdot \sum_{\sigma'\in\cF(\sigma)}\vol^{\dim(\sigma')}(\psi|_{\sigma'})
        \end{equation}
        for some constant $b$ depending only on the dimension of $Z$.
    \end{enumerate}
    If $\psi$ is equivariant then $\overline{\psi}$ can be taken to be equivariant as well.
\end{lemma}
Let $A=\overline{\iota_1^{-1}\circ \phi}\from X_0\to X_1$ be the map obtained by applyinh Lemma~\ref{lem:quant-admissible} to the map $\psi=\iota_1^{-1}\circ \phi$.
We claim that $A$ is Lipschitz on $X_0^{(1)}$ with a Lipschitz constant independent of $n$ and that $\avol(A|_\sigma)\lesssim C n^{d+1}$ for all $d\ge 2$ and $\sigma\in \cF^d(X_0)$. The bound on the Lipschitz constant will help us achieve condition (1) in Lemma~\ref{lem:Qi-conds}, and the bound on $\avol(A|_\sigma)$ will help us achieve conditions (3) and (4).

Let 
$$M=\max_{\sigma\in \cF^d(X_0), d\ge 1} \vol^d_R(\phi|_\sigma),$$
where $\vol^d_R$ is $d$--volume with respect to the Riemannian metric $d_R$ on $\H$. 
This maximum exists because there are only finitely many cells of $\Hz\backslash X_0$.
We can bound the number of uncollapsed discs of $A$ in terms of $M$ and $n$. Let $\ell_0$ and $\ell_1=n^{-1}\ell_0$ be length with respect to $d_0$ (the standard metric) and $d_1$, and let $\vol^d_0$ and $\vol^d_1$ denote $d$--volume with respect to $d_0$ and $d_1$.

Let $e$ be an edge of $X_0$. Then $A|_e$ is an edge path in $X_1$. By \eqref{eq:quant-avol-bound},
$$\ell_1(A|_e) = n^{-1} \avol(A|_e) \le n^{-1} b \big(\ell_0(\iota_1^{-1}\circ \phi|_e) + 2\big).$$
Since $\phi$ is horizontal on each edge $e$ of $X_0$ and $\iota$ is bilipschitz,
$$\ell_0(\iota_1^{-1}\circ \phi|_e) = \ell_0(\iota^{-1}\circ \delta_n\circ  \phi|_e) \approx \ell_R(\delta_n\circ \phi|_e) = n\ell_R(\phi|_e) \le nM.$$
Therefore, 
$$\ell_1(A|_e) \le n^{-1} b(nM+2) \le b M + 2bn^{-1},$$
so $A|_{X_0^{(1)}}$ is $((M+2)b)$--Lipschitz.

Recall from Section~\ref{sec:Carnot} that $\delta_n$ distorts $2$--dimensional area by a factor of at most $n^3$ and distorts $3$--dimensional volume by a factor of $n^4$. Thus if $d\ge 2$,
$$\vol^d_0(\iota_1^{-1}\circ \phi|_\sigma) \approx \vol_{R}^d(\delta_n\circ \phi|_\sigma) \le n^{d+1}M.$$
It follows that there is a $C=C(X,\phi)>1$ such that $\Lip(A|_{X_0^{(1)}})<C$ and 
\begin{equation}\label{eq:avol-A}
  \avol(A|_\sigma) \le Cn^{d+1}
\end{equation}
for any $d\ge 2$ and any $\sigma\in \cF^d(X_0)$.

We construct $Q$ by adjusting $A$ to make the uncollapsed balls into regular simplices. We use the following lemma, which will be proved in Appendix~\ref{app:admissible-maps}. A \emph{collared ball} in a $d$--dimensional manifold $M$ is the image of the radius--$\frac{1}{2}$ Euclidean $d$--ball $B(0,\frac{1}{2})\subset \R^d$ under an embedding of the unit ball $B(0,1)$.
\begin{lemma}\label{lem:normalizing uncollapsed}
  Let $Y$, $Z$ be finite dimensional simplicial complexes equipped with a multiple of the standard metric, and let $f\from Y\to Z$ be an admissible map. 
  For each simplex $\sigma\in \cF^d(Y)$ with $d\ge 2$, let $B^\sigma_1,\dots, B^\sigma_{\avol(f|_\sigma)}\subset \sigma$ be the uncollapsed balls of $f|_\sigma$ and let $\tau^\sigma_i=f(B^\sigma_i)\in \cF^d(Z)$ be the corresponding cells of $Z$.
  
  Let $C^\sigma_1,\dots, C^\sigma_{\avol(f|_\sigma)}\subset \sigma$ be disjoint collared balls and let $g^\sigma_i\from C^\sigma_i\to \tau^\sigma_i$ be homeomorphisms with the same orientation as $f|_{B^\sigma_i}$. Then $f$ is homotopic to an admissible map $h\from Y\to Z$ such that:
  \begin{enumerate}
    \item $h$ agrees with $f$ on $Y^{(1)}$.
    \item For each simplex $\sigma\in \cF^d(Y)$ with $d\ge 2$, the uncollapsed balls of $h|_\sigma$ are the $C^\sigma_i$'s, and for each $i$, $h$ agrees with $g^\sigma_i$ on $C^\sigma_i$.
    \item For every simplex $\sigma\in \cF^d(Y)$, $h(\sigma)\subset \supp f(\sigma)$.
  \end{enumerate}
  If $f$ is Lipschitz and the $g^\sigma_i$'s are Lipschitz, we can take $h$ to be Lipschitz on each cell of $Y$; if $f$ and the $g^\sigma_i$'s are equivariant, we can take $h$ to be equivariant as well.
\end{lemma}

We use these lemmas to prove Lemma~\ref{lem:Qi-conds}.
\begin{proof}[{Proof of Lemma~\ref{lem:Qi-conds}}] 
  We first construct $Q$. Let $c>0$ be a large number to be chosen later. For each simplex $\sigma\in \cF^d(X_0)$ with $d\ge 2$, let $B^\sigma_1,\dots, B^\sigma_{\avol(A|_\sigma)}$ be the uncollapsed balls of $A|_\sigma$ and let $\tau^\sigma_i=A(B^\sigma_i)\in \cF^d(X_1)$.

  Let $C>1$ be as in the remarks after Lemma~\ref{lem:quant-admissible}, so that $\Lip(A|_{X_0^{(1)}})<C$ and $\avol(A|_\sigma)\le Cn^{d+1}$ for each $d=2,3$ and each simplex $\sigma\in\cF^d(X_0)$. 
  Let $\beta>0$ be small enough that for $d=2,3$ and for any $N>0$, one can construct $N$ disjoint regular simplices of edge length $\beta N^{-\frac{1}{d}}$ in the interior of the unit regular simplex $\Delta^d$ (for instance, by inscribing a cube in $\Delta^d$, subdividing the cube into a grid with $\lceil N^{-\frac{1}{d}}\rceil$ grid cells on an edge, then inscribing a regular simplex in each grid cell).
  
  Let $c_0= \beta^{-1} C^{\frac{1}{2}}$. For any $c>c_0$ and $d=2,3$,
  $$a=c^{-1} n^{-\frac{3}{2}} < \beta C^{-\frac{1}{2}} n^{-\frac{3}{2}} \le \beta (Cn^{d+1})^{-\frac{1}{d}}.$$
  Thus for each $d=2,3$ and each cell $\sigma\in \cF^d(X_0)$, we can construct
  $$C_1^\sigma, \dots, C^\sigma_{\avol(A|_\sigma)}\subset \inter \sigma$$
  which are disjoint regular simplices of edge length $a$. For each $\sigma$ and $i$, let $g^\sigma_i\from C^\sigma_i \to \tau^\sigma_i$ be a similarity with the same orientation as $A|_{B^\sigma_i}$ and scale factor $\frac{\diam \tau^\sigma_i}{\diam C^\sigma_i}=\frac{s}{a}$.
  
  Let $Q\from X_0\to X_1$ be a Lipschitz, $\Hz$--equivariant, admissible map  satisfying Lemma~\ref{lem:normalizing uncollapsed} for this choice of $C^\sigma_i$ and $g^\sigma_i$. Then $Q_0=Q$ satisfies conditions (3) and (4) of Lemma~\ref{lem:Qi-conds}. Since $Q_i$ is the same map as $Q$, with the metric on the domain and range scaled by $s^i$, $Q_i$ satisfies conditions (3) and (4) for any $i$.
  
  We claim that $Q_i$ satisfies the rest of the desired conditions. Since $d_i=s^id_0$ for all $i$, we have $\Lip(Q)=\Lip(Q_i)$ for all $i$. By Lemma~\ref{lem:normalizing uncollapsed}, $Q|_{X_0^{(1)}} = A|_{X_0^{(1)}}$, so 
  $\Lip(Q_i|_{X_i^{(1)}}) < C$ by our choice of $C$.
  Thus, (1) holds when $c>C$.
  
  Condition (2) claims that 
  $d_c(\iota_0(x),\iota_1(Q(x)))\lesssim 1$
  for all $x\in X_0$.
  Let $x\in X_0$ and let $\sigma$ be a simplex of $X_0$ that contains $x$. Then
  $$d_c(\iota_0(x),\iota_1(Q(x))) \le d_c(\iota_0(x),\phi(x)) + d_c(\phi(\sigma),\iota_1(Q(x))) + \diam \phi(\sigma).$$
  By \eqref{eq:iota-kappa}, $d_c(\iota_0(x),\phi(x))< \frac{\kappa}{2}\lesssim 1$. Likewise, $$\diam \phi(\sigma) \le \max_{\delta\in \cF(X_0)} \diam \phi(\delta) \lesssim 1.$$
  Finally, we have 
  $$Q(x)\in Q(\sigma)\subset \supp A(\sigma) \subset \supp \supp(\iota_1^{-1}(\phi(\sigma))) = \supp(\iota_1^{-1}(\phi(\sigma))).$$
  Since simplices of $X_1$ have side length $s$,
  $$d_1(Q(x),\iota_1^{-1}(\phi(\sigma)))\le s.$$
  By \eqref{eq:Xi-qi},
  $$d_c(\iota_1(Q(x)),\phi(\sigma))\le 2ks \lesssim 1.$$
  Summing these inequalities, we find that there is a universal constant $c_2>0$  such that $d_c(\iota_0(x),\iota_1(Q(x)))\le c_2$. Then condition (2) holds for all $c>c_2$. 
  
  In fact, $\iota_0$ and $\iota_1\circ Q\from X_0\to \H$ are homotopic by a straight-line homotopy that moves each point distance at most $c_2$. Since $\iota_0$ is a degree--1 map, so is $\iota_1\circ Q$. Thus $Q$ is a surjective, degree--$1$ map.
  
  Finally, recall that by the construction of $\iota_0$, there is a $\kappa>0$ such that for every $\sigma\in \cF^3(X_0)$, there is a point $q\in \H$ such that if $y\in X_0\setminus \inter \sigma$, then $d_c(q,\iota_0(y))>\kappa$. By \eqref{eq:iota-kappa}, 
  \begin{equation}\label{eq:phi-far}
  d_c(q,\phi(y))>\frac{\kappa}{2}.
  \end{equation}
  Let $r=\frac{\kappa}{4}$, let
  $$D=\max_{\alpha\in \cF(X_0)} \diam \iota_0(\alpha)$$
  and let $n_0=\frac{4D}{\kappa}.$ Suppose that $n>n_0$.
  
  Suppose that $\tau\in \cF^3(X_0)$ and $\tau \ne \sigma$. Since $\tau\subset X_0\setminus \inter \sigma$, we have $d_c(q,\phi(\tau))>\frac{\kappa}{2}$, and we claim that $d_c(q,\iota_1(Q(\tau))) > r$.
  By Lemma~\ref{lem:quant-admissible} and Lemma~\ref{lem:normalizing uncollapsed},
  $$Q(\tau)\subset \supp A(\tau) = \supp \left(\overline{\iota_1^{-1}\circ \phi}(\tau)\right) \subset \supp \left(\iota_1^{-1}(\phi(\tau))\right).$$
  That is, for any $w\in Q(\tau)$, there are $z\in \tau$ and $\alpha\in \cF(X_1)$ such that $\{w, \iota_1^{-1}(\phi(z))\} \subset \alpha$ and thus $$d_c(\iota_1(w), \phi(z)) \le \diam \iota_1(\alpha) \le Dn^{-1}\le \frac{\kappa}{4}.$$
  By the triangle inequality,
  $$d_c(q, \iota_1(w)) \ge d_c(q,\phi(z)) - d_c(\phi(z), \iota_1(w)) > \frac{\kappa}{4} = r.$$
  Therefore, $d_c(q,\iota_1(Q(\tau))) > r$. Since this holds for every $\tau\ne \sigma$, we have 
  $$d_c(q,\iota_1(Q(X_0\setminus \sigma)))>r.$$ This proves condition (5).
\end{proof}

\section{Hölder bounds for $P$}\label{sec:holder-bounds}

Let $c$ be as in Lemma~\ref{lem:Qi-conds}.
By condition (2) of Lemma~\ref{lem:Qi-conds}, for any $x\in X_0$, we have $d_c(\iota_i(P_i(x)),\iota_{i+1}(P_{i+1}(x))) < cs^i,$
so the $P_i$'s converge uniformly. Let $P=\lim_i \iota_i\circ P_i$. Then
\begin{equation}\label{eq:unif Cauchy}
  d_c\big(\iota_i(P_i(x)),P(x)\big) < \frac{c s^i}{1-s} \le 2cs^i.
\end{equation}

In this section, we will prove an exponential bound on $\Lip(P_i)$ and combine it with \eqref{eq:unif Cauchy} to prove Hölder bounds on $P$. We bound $\Lip(P_i)$ as follows.
\begin{lemma}\label{lem:lip growth}
  Let $c$, $n$, $a=c^{-1}n^{-\frac{3}{2}}$, and $Q_i\from X_i\to X_{i+1}$ be as in Lemma~\ref{lem:Qi-conds}. Let $P_i=Q_{i-1}\circ \dots \circ Q_0$. There is a $C>0$ depending on $X$ and $n$ such that $\Lip(P_i)\le C \cdot (\frac{s}{a})^i$.
\end{lemma}

\begin{proof}
  Let $\lambda_i\from X_0\to \R$ be the local Lipschitz constant
  $$\lambda_i(x)=\lim_{r\to 0^+} \Lip(P_i|_{B(x,r)}),$$
  where $B(x,r)$ denotes the ball with respect to $d_0$. Let $L=\Lip(Q_1)$; since the $Q_i$'s are scalings of $Q_1$, we have $L=\Lip(Q_i)$ for all $i$. Since $Q_i$ scales uncollapsed balls by $\frac{s}{a}$, we have $\frac{s}{a}\le L$, but $\frac{s}{a}$ will typically be much smaller than $L$. In general, $\lambda_{i+1}(x)\le L \lambda_i(x)$, and we will prove the lemma by showing that $\lambda_{i+1}(x)\le \frac{s}{a}\lambda_i(x)$ for all but six values of $i$.
  
  For each $i$, define $D_i \from X_0\to \{0,1,2,3\}$ so that $D_i(x)$ is the dimension of the smallest cell of $X_i$ containing $P_i(x)$. Let
  $$\widehat{D}_i(x)=\lim_{r\to 0^+} \max D_i(B(x,r)).$$
  Then $\widehat{D}_0(x)=3$ and $\widehat{D}_i(x)\ge D_i(x)$ for all $x$ and $i$. 
  Since $Q_i$ is a cellular map, we have $D_{i+1}(x)\le D_i(x)$ and $\widehat{D}_{i+1}(x)\le \widehat{D}_i(x)$ for all $x$ and $i$.

  Let $x\in X_0$. We claim that $\widehat{D}_i(x)=D_i(x)=D_{i+1}(x)$ for all but at most six values of $i$ and that if $\widehat{D}_i(x)=D_i(x)=D_{i+1}(x)$, then $\lambda_{i+1}(x)\le \frac{s}{a}\lambda_i(x)$.
  
%   The first claim follows from the fact that
%   $$\widehat{D}_0(x), D_0(x), \widehat{D}_1(x), D_1(x), \widehat{D}_2(x), D_2(x),\dots$$
%   is a non-increasing sequence. We have already seen that $\widehat{D}_i(x)\ge D_i(x)$, so it suffices to show that $D_i(x)\ge \widehat{D}_{i+1}(x)$ for all $i$. Let $d=D_i(x)$. Then $P_i(x)\in X_i^{(d)}$, and by condition (3) of Lemma~\ref{lem:Qi-conds}, there is an $r>0$ such that $P_{i}(B(x,r))$ does not intersect any uncollapsed $(d+1)$--ball and thus $P_{i+1}(B(x,r))\subset X_{i+1}^{(d)}$, i.e.\ $\widehat{D}_{i+1}(x)\le d=D_i(x)$. Since $\widehat{D}_0(x)=3$, there are at most six values of $i$ such that $D_i(x)\ne D_{i+1}(x)$ or $\widehat{D}_i(x)\ne D_i(x)$.
  
  First, since $(D_i(x))_i$ is non-increasing, there are at most three values of $i$ such that $D_i(x)\ne D_{i+1}(x)$. 
  Suppose that $\widehat{D}_i(x)\ne D_i(x)$; we must have $\widehat{D}_{i}(x) > {D}_i(x)$. Let $d=D_i(x)$, so that $P_i(x)\in X_i^{(d)}$.
  By condition (3) of Lemma~\ref{lem:Qi-conds}, there is an $r>0$ such that $P_{i}(B(x,r))$ does not intersect any uncollapsed $(d+1)$--ball and thus $P_{i+1}(B(x,r))\subset X_{i+1}^{(d)}$, i.e., $\widehat{D}_{i+1}(x)\le D_i(x) < \widehat{D}_{i}(x)$. Since $(\widehat{D}_i(x))_i$ is non-increasing, this can happen for at most three values of $i$. There are thus at most six values of $i$ such that $D_i(x)\ne D_{i+1}(x)$ or $\widehat{D}_i(x)\ne D_i(x)$.
  
  % In fact, this shows that $$\widehat{D}_0(x), D_0(x), \widehat{D}_1(x), D_1(x), \widehat{D}_2(x), D_2(x),\dots$$ is a non-increasing sequence.
  
  For every other value of $i$, we have $\widehat{D}_i(x)=D_i(x)=D_{i+1}(x)$. If $\widehat{D}_i(x)\le 1$, then there is some $r>0$ such that $P_i(B(x,r))\subset X_i^{(1)}$. By condition (1) of Lemma~\ref{lem:Qi-conds}, $\Lip(Q_i|_{P_i(B(x,r))}) < c \le \frac{s}{a}$, so $\lambda_{i+1}(x)\le\frac{s}{a}\lambda_{i}(x)$.

  Otherwise, if $\widehat{D}_i(x) > 1$, let $d=D_i(x)$. Since $D_i(x)=D_{i+1}(x)$, $P_i(x)$ lies in the interior of some uncollapsed $d$--ball $B$ of $Q_i(x)$. Since $\widehat{D}_i(x)=d$, there is some $r>0$ such that $P_i(B(x,r))\subset X^{(d)}_i$. Thus, if $r>0$ is sufficiently small, we have $P_i(B(x,r))\subset B$. Since $Q_i$ rescales $B$ by a factor of $\frac{s}{a}$, we have $\lambda_{i+1}(x)=\lambda_{i}(x) \frac{s}{a}$.

  Thus, for up to six values of $i$, we have $\lambda_{i+1}(x)\le L \lambda_{i}(x)$, and for all the rest, $\lambda_{i+1}(x)\le  \frac{s}{a} \lambda_{i}(x)$. Since $\lambda_0(x)=1$, this implies
  $$\lambda_i(x)\le L^6 \left(\frac{s}{a}\right)^i$$
  for all $x$. Since $X_0$ is a path metric space, 
  $$\Lip(P_i) \le \sup_{x\in X_0} \lambda_i(x) \le L^6 \left(\frac{s}{a}\right)^i,$$
  as desired.
\end{proof}

Given the lemma, one can use \eqref{eq:unif Cauchy} to show that $P$ is locally Hölder and prove Theorem~\ref{thm:3d equivariant}. This is the same argument as in Lemma~\ref{lem:1+e-holder}.
\begin{proof}[{Proof of Theorem~\ref{thm:3d equivariant}}]
  Let $\alpha<\frac{2}{3}$ and let $c>1$ be as in Lemma~\ref{lem:Qi-conds}.
  Since $a=c^{-1}n^{-\frac{3}{2}}$ and $s=n^{-1}$,
  $$\lim_{n\to \infty}\frac{\log s}{\log a} = \lim_{n\to \infty}\frac{\log n^{-1}}{\log c^{-1}n^{-\frac{3}{2}}} = \frac{2}{3}.$$
  Let $n$ be large enough that $\frac{\log s}{\log a}\ge \alpha$. 
  Let $k>1$ be the quasi-isometry constant of $\iota$ as in \eqref{eq:X-qi} and \eqref{eq:Xi-qi}, and let $C$ be as in Lemma~\ref{lem:lip growth} so that $\Lip(P_i)\le C\cdot (\frac{s}{a})^i$.

  We claim that $P$ is locally $(K,\alpha)$--Hölder, where $K=\frac{4c+kC+k}{s}$.
  Let $x,y\in X_0$ such that $0<d_0(x,y)<1$. Let $i\ge 0$ be the integer such that $a^{i+1}\le d_0(x,y) < a^i$. Then, by \eqref{eq:unif Cauchy} and \eqref{eq:Xi-qi},
%   \begin{align*}
%     d_c(P(x),P(y)) & \le 4cs^i + d_c(\iota_i(P_i(x)), \iota_i(P_i(y))) \\
%     & \le 4cs^i + k \Lip(P_i) a^i + ks^i \\
%     & \le \frac{4c+k(C+1)}{s}s^{i+1}\\
%     & \le K d_0(x,y)^{\frac{\log s}{\log a}}\\
%     & \le K d_0(x,y)^\alpha,
%   \end{align*}
  \begin{multline*}
    d_c(P(x),P(y)) \le 4cs^i + d_c(\iota_i(P_i(x)), \iota_i(P_i(y)))\\ \le 4cs^i + k \Lip(P_i) a^i + ks^i 
    \le K s^{i+1}
%    \le K d_0(x,y)^{\frac{\log s}{\log a}}     
    \le K d_0(x,y)^\alpha,
  \end{multline*}
  as desired.

  Since $\iota_0$ is a bilipschitz equivalence from $(\H,d_R)$ to $X_0$, the composition $g=P\circ \iota_0^{-1}\from (\H,d_R)\to (\H,d_c)$ is an $\Hz$--equivariant, locally $\alpha$--Hölder map. Since $\H$ is contractible, any such map is equivariantly homotopic to the identity.  
\end{proof}

We prove Theorem~\ref{thm:approximations} by using $g$ to construct a sequence of locally H\"older maps that converge to the identity map.
\begin{proof}[{Proof of Theorem~\ref{thm:approximations}}]
  Let $Y$ be a compact metric space and let $\varphi \from Y \to \H$ be a continuous map.  Let $\epsilon>0$ and let $0<\alpha<\frac{2}{3}$. 

  We first approximate $\varphi$ by a Lipschitz map.  Let $\delta>0$ be such that if $x,y\in \H$ and $d_R(x,y)<\delta$, then $d_c(x,y)<\frac{\epsilon}{2}$.  Let $B\subset \H$ be a ball containing $\varphi(Y)$ and note that the Riemannian metric on $B$ is bilipschitz equivalent to the Euclidean metric on $B$ (with a constant depending on the radius of $B$).  Lipschitz maps from $Y$ to $\R^3$ are dense in the space of continuous maps by \cite[6.8]{Heinonen} or by the Stone-Weierstrass theorem, so there is a Lipschitz map $\lambda\from Y\to (\H,d_R)$ such that $d_R(\varphi(y),\lambda(y))<\delta$ for every $y\in Y$ and hence $d_c(\varphi(y), \lambda(y))<\frac{\epsilon}{2}$.

  By Theorem~\ref{thm:3d equivariant}, there is an $\Hz$--equivariant map $g\from (\H,d_R) \to (\H,d_c)$ that is locally $\alpha$--H\"older.  The $\Hz$--equivariance of $g$ implies that there is an $m>0$ such that $m\ge d_c(x,g(x))$ for all $x\in \H$.  Let $r=\frac{\epsilon}{2m}$.  For any $x\in \H$,
  $$d_c\big(x,(\delta_r\circ g\circ\delta_r^{-1})(x)\big)= r\cdot d_c\big(\delta_r^{-1}(x), g(\delta_r^{-1}(x))\big) \le rm =\frac{\epsilon}{2}.$$

  Let $\psi=\delta_r\circ g \circ \delta_r^{-1} \circ \lambda$.  Then $\psi$ is locally $\alpha$--H\"older with respect to $d_c$, and since $Y$ is compact, a locally $\alpha$--H\"older map from $Y$ to $\H$ is globally $\alpha$--H\"older.  For any $y\in Y$,
  \begin{multline*}
    d_c(\varphi(y),\psi(y))\le d_c(\varphi(y),\lambda(y))+d_c(\lambda(y),\psi(y)) \\ \le \frac{\epsilon}{2} + d_c\big(\lambda(y), (\delta_r\circ g\circ \delta_r^{-1})(\lambda(y))\big) \le \epsilon,
  \end{multline*}
  so $\psi$ is an $\alpha$--H\"older map that is $\epsilon$--close to $\varphi$. 
\end{proof}

\section{A globally H\"older map from $\R^3$ to $(\H, d_c)$}\label{sec:Hoelder-R3-Hc}

In this section, we prove Theorem~\ref{thm:3d tangent cone} by constructing a self-similar globally H\"older map from $\R^3$ to $\H$. 

First, we show that the construction in Section~\ref{sec:cellulation-and-Q} produces a self-similar map. 
Recall that all the complexes $X_i$ have the same underlying simplicial complex $X$ and that all of the maps $Q\from X_0\to X_1$ and $Q_i\from X_i\to X_{i+1}$ are the same as maps from $X$ to $X$.
The main difference between the $X_i$'s is that they are equipped with different metrics $d_i=s^i d_\Delta$, different actions $\zeta_i[g]=\zeta[\delta_{n^i}(g)]$ (where $\zeta[g]\from X\to X$ is the action of $\Hz$ on $X$), and different equivariant maps $\iota_i\from X\to \H$, $\iota_i=\delta_{n^{-i}}\circ \iota$ (where $\iota\from X\to \H$ is an equivariant homeomorphism). 

In this section, we will identify all of the $X_i$'s with  $X$. Under this identification, 
$$P_i=Q_{i-1}\circ \dots \circ Q_0 = Q^i.$$

\begin{lemma}\label{lem:P-self-sim}
  With notation as in Section~\ref{sec:cellulation-and-Q}, let $\sigma\in \cF^3(X)$. Let $D\subset X$ be an uncollapsed ball of $Q$ and suppose that there is a $g\in \Hz$ such that $Q(D)=\zeta[g^{-1}](\sigma)$. Let $\tau=Q(D)$.
  
  With respect to the standard metric $d_\Delta$, $D$ is a regular simplex with side length $a$ and $\sigma$ is a regular simplex with side length $1$, and there is a similarity $h\from (D, d_\Delta) \to (\sigma, d_\Delta)$ with scaling factor $a^{-1}>1$ and a Heisenberg similarity $m\from \H \to \H$, $m(x)=g\delta_n(x)$ such that:
  \begin{equation}\label{eq:self-sim-zoom-in}
    P(x)=m(P(h^{-1}(x))) \qquad\text{for all } x\in \sigma,
  \end{equation}
  and
  \begin{equation}\label{eq:self-sim-zoom-out}
    P(y)=m^{-1}(P(\zeta[g](Q(y))))\qquad \text{for all } y\in X.
  \end{equation}
\end{lemma}
\begin{proof}
  
%  Let $\sigma\in \cF^3(X)$. By Lemma~\ref{lem:Qi-conds}, there are $r>0$ and a $q\in X$ such that if $y\in X\setminus \sigma$, then $d_1(q,Q(y)) > r$, i.e., $d_\Delta(q,Q(y)) > r n$. If $n$ is large enough, then the ball of radius $r n$ around $q$ contains a fundamental domain for $\zeta$; in particular, if $n$ is large enough, there is a $g\in \Hz$ and a simplex $\tau$ such that $\sigma=g\tau$ and $\tau \subset B(q,r n;d_\Delta)$. 
  
%  Since $Q$ is surjective and degree--$1$, we have $\tau \subset Q(\sigma)$. Let $D\subset \sigma$ be an uncollapsed ball of $Q$ such that $Q$ sends $D$ to $\tau$ by an orientation-preserving similarity. Then $D$ is a regular simplex of side length $a$ with respect to $d_\Delta$.
  
  Let $m(x)=g\delta_n(x)$. Since $P=\lim_{i\to \infty} \iota_i\circ Q^i$ and $P$ is equivariant,
  \begin{multline*}
    m(P(y))=\lim_{i\to \infty} g\delta_{n}(\iota_i(Q^i(y))) = \\
    \lim_{i\to \infty} g\iota_{i-1}(Q^{i-1}(Q(y))) = g P(Q(y)) = P(\zeta[g](Q(y))).
  \end{multline*}
  Applying $m^{-1}$ to both sides proves \eqref{eq:self-sim-zoom-out}.

  Equation \eqref{eq:self-sim-zoom-in} follows from \eqref{eq:self-sim-zoom-out}.
  By condition (3) of Lemma~\ref{lem:Qi-conds}, $D$ is a regular simplex with sides of $d_\Delta$--length $a$ and $Q$ sends $D$ to $\tau$ by a similarity that scales the $d_\Delta$--metric by $a^{-1}$.
  Let $h\from D\to \sigma$ be the similarity $h=\zeta[g] \circ Q|_D\from D\to \sigma$. Let $x\in \sigma$ and let $y=h^{-1}(x)$. By \eqref{eq:self-sim-zoom-out},
  $$P(x)=P(\zeta[g](Q(y)))=m(P(y))=m(P(h^{-1}(x))),$$
  which proves \eqref{eq:self-sim-zoom-in}.
  
%   Recall that $Q$ is equivariant as a map $X_0\to X_1$, i.e.,
%   $$Q(\zeta[g](x))=\zeta[\delta_n(g)](Q(x))$$
%   for all $g\in \Hz$, $x\in X$. Therefore, for any $x\in D$,
%   $$Q^i(h(x)) = Q^i(\zeta[g](Q(x)))=\zeta[\delta_{n^i}(g)](Q^{i+1}(x))= \zeta_i[g](Q^{i+1}(x)).$$
%   Since $\iota_i\from (X,\zeta_i) \to \H$ is equivariant,
%   $$\iota_i(Q^i(h(x))) = \iota_i(\zeta_i[g](Q^{i+1}(x)) = g\iota_i(Q^{i+1}(x)) = g\delta_n(\iota_{i+1}(Q^{i+1}(x))).$$
%   Letting $i\to \infty$ and using the fact that $\lim_i \iota_i\circ Q^i= \lim_i \iota_i\circ P_i= P$, we get 
%   $$P(h(x))=g\delta_n(P(x))$$
%   for all $x\in D$. Let $m(x)=g\delta_n(x)$; this is a Heisenberg similarity with scale factor $n$, and $m(P(h^{-1}(x))) = P(x)$ for all $x\in \sigma$. This proves \eqref{eq:self-sim-zoom-in}.
\end{proof}

If we can find simplices $\sigma$ satisfying Lemma~\ref{lem:P-self-sim}, we can use them to construct self-similar Hölder maps $F\from \R^3\to \H$, i.e., maps such that there are expanding similarities $m\from \H\to \H$ and $h\from \R^3 \to \R^3$ that satisfy $m(F(h^{-1}(x)))=F(x)$ for all $x\in \R^3$. In order to prove Theorem~\ref{thm:3d tangent cone}, however, we want $F$ to be proper, i.e., $F^{-1}(K)$ is compact for every compact $K\subset \H$. 

Since $m$ is expanding, it has a unique fixed point $q\in \H$. If $F$ is proper and $F(p)=q$, then $F(h(p)) = m(F(p)) = q$, so $h(p)\in F^{-1}(q)$, i.e., $F^{-1}(q)$ is invariant under $h$. But the only compact sets that are invariant under $h$ are the empty set and the set $\{p\}$ where $p$ is the unique fixed point of $h$. Thus, in order to use Lemma~\ref{lem:P-self-sim} to construct a proper self-similar map $F\from \R^3\to \H$, we will need to adjust the construction so that there is a point $q\in \H$ such that $P^{-1}(q)$ is a single point.

We need the following lemma (see also \cite{WhiteMappingsThatMinimize}).
\begin{lemma}\label{lem:hurewicz fillings}
  Let $d\ge 3$, let $Y$ be a $(d-2)$--connected simplicial complex, and let $f\from D^d\to Y$ be a continuous map such that $f(\partial D^d)\subset Y^{(d-1)}$ and $f(D^d)\subset Y^{(d)}$.  For each $d$--cell $\delta\in \cF^d(Y)$, the degree of $f$ is the same at every point in the interior of $\delta$; let $\deg_\delta(f)$ be this degree.  Then there is an admissible map $g\from D^d\to Y$ which agrees with $f$ on $\partial D^d$ and such that for every $d$--cell $\delta\in \cF^d(Y)$, the number of uncollapsed balls in $g^{-1}(\delta)$ is $|\deg_\delta(f)|$.  
\end{lemma}
\begin{proof}
  The image $f([D^d])$ of the fundamental class of $D^d$ is a cellular $d$--chain in $Y$ that can be written
  $$f([D^d])=\sum_{\delta\in \cF^d(Y)}\deg_\delta(f) [\delta],$$
  where $[\delta]$ represents the fundamental class of $\delta$, and $\partial f([D^d])=f([\partial D^d])$.

  Let $k=\sum_{\delta} |\deg_\delta(f)|$ and let $B_1,\dots, B_k\subset \inter D^d$ be $k$ disjoint smoothly embedded balls.  For each ball, choose a cell $\delta_i\in \cF^d(Y)$ and a homeomorphism $g_i\from B_i\to \delta_i$ so that each cell $\delta$ is chosen $|\deg_\delta(f)|$ times and so that the orientation of $g_i$ corresponds to the sign of $\deg_\delta(f)$.  

  We will extend the maps $g_i$ to the desired map $g$.  We proceed as in \cite{WY-LipHom}.  Consider the complement $E=D^d\setminus \bigcup_i \inter(B_i)$.  Choose $v\in \partial D^d$ and for each $i$, let $v_i\in \partial B_i$.  For each $i$, let $\gamma_i$ be a simple smooth curve connecting $v$ to $v_i$, and suppose that the $\gamma_i$'s are disjoint.  The interior of the complement $E\setminus \bigcup_i \gamma_i$ is homeomorphic to an open $d$--ball, and we can give $E$ the structure of a CW complex with vertices $v,v_1,\dots, v_k$; edges $\gamma_1,\dots, \gamma_k$; $(d-1)$--cells $\partial D^d$, $\partial B_1,\dots, \partial B_k$; and a single $d$--cell, which we call $\sigma$.  Define $g$ on $E^{(d-1)}$ so that it agrees with $g$ on $\partial D^d$ and with $g_i$ on each $\partial B_i$; since $Y$ is connected, we can extend $g$ on each edge.  It remains to extend $g$ on $\sigma$.

  Since $Y$ is $(d-2)$--connected, so is $Y^{(d-1)}$, and Hurewicz's Theorem implies that $\pi_{d-1}(Y^{(d-1)})\cong H_{d-1}(Y^{(d-1)})$.  Let $\alpha\from S^{d-1} \to E^{(d-1)}$ be the attaching map of $\sigma$.  Then
  \begin{align*}
    g(\alpha([S^{d-1}]))
    &=f([\partial D^d])-\sum_i g([\partial B_i])\\
    &=\sum_{\delta} \deg_{\delta}(f) [\partial \delta] -\sum_{\delta} \deg_{\delta}(f) [\partial \delta]\\
    &=0.
  \end{align*}
  Therefore, $g\circ\alpha$ is null-homotopic, and the obstruction to extending $g$ to a map from $D^d$ to $Y^{(d-1)}$ vanishes.
\end{proof}

Applying this lemma to $Q$ produces a map that satisfies Lemma~\ref{lem:Qi-conds} and also has a simplex with a small preimage.
\begin{lemma}\label{lem:admissible Q degree}
  For any sufficiently large $n$, there is an admissible Lipschitz map $Q\from X\to X$ which is $\Hz$--equivariant as a map from $X_0$ to $X_1$, satisfies Lemma~\ref{lem:Qi-conds}, and satisfies the following conditions.
  \begin{itemize}
  \item For every $\sigma\in \cF^3(X)$, there is a $\tau\in \cF^3(X)$ such that $\sigma=\zeta[g](\tau)$ for some $g\in \Hz$ and $Q^{-1}(\inter \tau)$ is the interior of a single uncollapsed ball in $\sigma$. 
  \item Let $r>0$ be as in Lemma~\ref{lem:Qi-conds}. For every $y\in X \setminus \inter \sigma$, 
  $$d_c(P(y),\iota_1(\tau)) \ge \frac{r}{2}.$$
  \end{itemize} 
\end{lemma}
\begin{proof}
  By Lemma~\ref{lem:hurewicz fillings}, for any sufficiently large $n$, there is a map $Q$ satisfying Lemma~\ref{lem:Qi-conds} with the property that for every $\sigma\in \cF^3(X)$ and every $\delta\in \cF^3(X)$, exactly $\big|\deg_\delta(Q|_\sigma)\big|$ uncollapsed balls of $Q|_\sigma$ map to $\delta$. 
  
  Let $\sigma\in \cF^3(X)$. By Lemma~\ref{lem:Qi-conds}, there are $r>0$ and $q\in \H$ such that if $y\in X\setminus \sigma$, then $d_c(q,\iota_{1}(Q(y)))>r$.
  Suppose that $n>\frac{8c}{r}$ and that $n$ is large enough that for any $d_c$--ball $B$ of radius $\frac{rn}{4}$ in $\H$, there is a $g\in \Hz$ such that $g\iota(\sigma)\subset B$. Then there is a $g\in \Hz$ such that 
  $$g\iota(\sigma)\subset B\bigg(\delta_n(q), \frac{rn}{4};d_c\bigg),$$
  where $B(q,r;d_c)$ is the metric ball with respect to $d_c$.
  Let $\tau=\zeta[g](\sigma)$. Then
  $$\iota_1(\tau) = 
  \delta_{n^{-1}}(g\iota(\sigma))\subset 
  B\left(q, \frac{r}{4};d_c\right),$$
  and by the triangle inequality, for every $y\in X\setminus \sigma$, 
  $$d_c(\iota_1(Q(y)),\iota_1(\tau)) > \frac{3r}{4},$$
  i.e., $Q(y)\not \in \tau$. Thus $Q^{-1}(\tau)\subset \sigma$. Since $Q$ is degree--1, we have $\deg_\tau(Q|_\sigma)=1$, so there is exactly one uncollapsed ball that maps to $\tau$.
  Finally, by \eqref{eq:unif Cauchy} and our choice of $n$, for every $y\in X\setminus \sigma$,
  $$d_c(P(y),\iota_1(\tau)) > \frac{3r}{4} - 2cs  > \frac{r}{2}.$$
  By continuity, if $y\in X\setminus \inter \sigma$, then $d_c(P(y),\iota_1(\tau)) \ge \frac{r}{2}$.
\end{proof}

Now we prove Theorem~\ref{thm:3d tangent cone}.
\begin{proof}[{Proof of Theorem~\ref{thm:3d tangent cone}}]
  Let $\sigma\in \cF^3(X)$ and let $\tau$ satisfy Lemma~\ref{lem:admissible Q degree} for some large $n$ to be chosen later. Let $D\subset \sigma$ be the uncollapsed ball of $Q$ such that $Q(D)=\tau$. Let $g\in \Hz$ be such that $\tau=\zeta[g^{-1}]\sigma$.

  By Lemma~\ref{lem:P-self-sim}, there are a similarity $h\from D\to \sigma$ with scale factor $a^{-1}$ and a similarity $m\from \H\to \H$, $m(x)=g\delta_n(x)$ such that $P(x)=m(P(h^{-1}(x)))$ for all $x\in \sigma$. 
  We identify $\sigma$ with a unit simplex in $\R^3$ and $D\subset\sigma$ with a subset of that simplex. Then we can extend $h$ to a similarity $h\from \R^3\to \R^3$ that sends $D$ to $\sigma$ and has scaling factor $n$. Since $n>1$, $h$ has a unique fixed point $p$ such that $\lim_{i\to \infty} h^{-i}(x)=p$ for every $x\in \R^3$; since $h^{-1}(\sigma)=D\Subset \sigma$, we have $p\in \inter D\subset \inter \sigma$.
  
  We define $F\from \R^3\to \H$ as follows. 
  Let $x\in \R^3$. Since $\lim_{i\to \infty} h^{-i}(x)=p$, there is an $i\in \Z$ such that $h^{-i}(x)\in \sigma$. Let
  $$F(x)=m^i(P(h^{-i}(x))).$$

  Then $F$ is well-defined; if $i<j$ and $h^{-i}(x)\in \sigma$, then Lemma~\ref{lem:P-self-sim} implies
  \begin{align*}
    m^{i}(P(h^{-i}(x)))
      =m^{i}(m^{j-i} \circ P \circ h^{i-j}(h^{-i}(x))) 
      =m^{j}(P(h^{-j}(x))).
  \end{align*}
  Furthermore, for all $x\in h^i(\sigma)$, we have $h^{-1}(x)\in h^{i-1}(\sigma)$, so
  $$m(F(h^{-1}(x))) = m(m^{i-1}(P(h^{-i+1}(h^{-1}(x))))) = m^{-i}(P(h^{i}(x))) = F(x).$$
  That is, $F$ is self-similar.

  It remains to show that $F$ is globally Hölder, proper, and degree $1$. 
  
  First, we show that $F$ is $\eta$--Hölder for $\eta:=\frac{\log s}{\log a}$. Since $P$ is locally $\eta$--Hölder, the restriction $P|_\sigma$ is $(C, \eta)$--H\"older for some constant $C$. Let $x,y\in \R^3$. Then there exists $i\in \Z$ such that $x,y\in h^i(\sigma)$ and
  \begin{align*}
    d_c(F(x),F(y)) &= d_c(m^{i}(P(h^{-i}(x))), m^{i}(P(h^{-i}(y))))\\
     & = n^i d_c(P(h^{-i}(x)), P(h^{-i}(y))) \\ 
     &\le C n^i |h^{-i}(x)-h^{-i}(y)|^\eta\\
     &\le C n^i (a^i |x-y|)^\eta\\
     & = C |x-y|^\eta.
  \end{align*}
  Thus $F$ is globally $\eta$--H\"older.  
  
  Next, we show $F$ is proper. Let $q:=F(p)=P(p)\in \H$. Then $q$ is the unique fixed point of $m$, because $m(q)=m(F(p))=F(h(p))=F(p)$. Let $r$ be as in Lemma~\ref{lem:admissible Q degree}, so that $d_c(P(y),\iota_1(\tau))>\frac{r}{2}$ for every $y\in X\setminus \sigma$. We will show that if $n$ is sufficiently large, and $z\in \R^3\setminus \sigma$, then $d_c(P(z),q) > \frac{r}{4}$. We will then use self-similarity to conclude that $F$ is proper.
  
  First, we claim that when $n$ is sufficiently large
  \begin{equation} \label{eq:degree-around-fixed}
    d_c(P(y),q) > \frac{r}{4}\qquad \text{for all $y\in X\setminus \inter \sigma$.}
  \end{equation}
  Let $c$ be as in Lemma~\ref{lem:Qi-conds} and Lemma~\ref{lem:admissible Q degree} and suppose that $n>\frac{8c}{r}$. Since $p\in D$, we have $Q(p)\in \tau$. On one hand, by Lemma~\ref{lem:admissible Q degree}, for any $y\in X\setminus \inter \sigma$, we have 
  $d_c(P(y),\iota_1(Q(p)))\ge \frac{r}{2}$; on the other hand, by \eqref{eq:unif Cauchy},
  $$d_c(q,\iota_1(Q(p))) = d_c(P(p), \iota_1(Q(p))) < 2cs < \frac{r}{4},$$
  so the triangle inequality implies \eqref{eq:degree-around-fixed}.
  
  Now suppose that $z \in \R^3\setminus \sigma$. Let $i>0$ be such that $h^{-i}(z)\in \sigma\setminus D$ and let $y=h^{-i}(z)$ so that $F(z)=m^i(P(h^{-i}(z)))=m^i(P(y))$.
%   $$F(z) = m^i(P(h^{-i}(z))) \in m^{i}(P(\sigma\setminus D)).$$
%   Let $x=h^{-1}(z)$. 
%  Since $m(a)=g\delta_n(a)$
  Since $y\in \sigma\setminus D$, we have $Q(y)\not \in \inter \tau$ and $\zeta[g](Q(y)) \not\in \inter \sigma$.
  By \eqref{eq:self-sim-zoom-out},
  $$P(y)=m^{-1}(P(\zeta[g](Q(y)))),$$
  so by \eqref{eq:degree-around-fixed},
  $$d_c(P(y), q) = n^{-1} d_c(m(P(y)),q) = n^{-1} d_c\big(P(\zeta[g](Q(y))),q\big) > \frac{r}{4n},$$
  and
  $$d_c(F(z),q) = n^i d_c(P(y),q) > \frac{r}{4},$$
  as desired.
  
  By the scale-invariance of $F$, if $z\in \R^3\setminus h^i(\sigma)$, then $d_c(F(z),q)>\frac{r}{4}n^i$. Let $K\subset \H$ be a compact set. Then there is some $i\in \Z$ such that $K\subset B(q,\frac{r}{4}n^i;d_c)$ and thus $F^{-1}(K)\subset h^i(\sigma)$. Since $F^{-1}(K)$ is closed and bounded, it is compact. Therefore, the preimage of any compact set is compact, and $F$ is proper.
 
  Finally, we claim that $F$ has degree $1$. It suffices to show that $\deg_q(F)=1$. Let $B=B(q,\frac{r}{4};d_c)$. Then $F^{-1}(B)\subset \sigma$, so
  $$\deg_q(F) = \deg_q(F|_\sigma) = \deg_q(P|_\sigma) = \deg_q(P)=1.$$
  Since $F$ is proper, the degree is constant on $\H$.
\end{proof}

\appendix

\section{Proofs of Lemmas~\ref{lem:quant-admissible} and \ref{lem:normalizing uncollapsed}}\label{app:admissible-maps}

The aim of the appendix is to prove Lemmas~\ref{lem:quant-admissible} and \ref{lem:normalizing uncollapsed}. We first show:

\begin{lemma}\label{lem:def-to-admissible}
 Let $Z$ be a finite dimensional simplicial complex equipped with the standard metric and let $\varrho\from D^d\to Z$ be a Lipschitz map such that $\varrho(\partial D^d)\subset Z^{(d-1)}$. Then $\varrho$ is homotopic relative to $\partial D^d$ to an admissible map $\overline{\varrho}\from D^d\to Z$ with $$\avol(\overline{\varrho})\leq c\cdot \vol^d(\varrho),$$ where $c$ only depends on the dimension of $Z$.  The homotopy between $\varrho$ and $\overline{\varrho}$ is Lipschitz with image in $\supp(\varrho(D^d))$ and its $(d+1)$--volume is bounded from above by $c\cdot \vol^d(\varrho)$. 
\end{lemma}
If $d=1$ then we may assume that $\overline{\varrho}$ is an edge path parametrized with constant speed. 

\begin{proof}
 By the (proof of the) Federer-Fleming deformation theorem, see for example \cite[Theorem 10.3.3]{Epstein-et-al-Wordprocess}, $\varrho$ is Lipschitz homotopic relative to $\partial D^d$ to a map $\hat{\varrho}\from D^d\to Z^{(d)}$ such that the $d$--volume of $\hat{\varrho}$ and the $(d+1)$--volume of the homotopy are both bounded from above by $c\cdot \vol^d(\varrho)$ for some constant $c$ only depending on the dimension of $Z$. Moreover, the image of the homotopy is contained in $\supp(\varrho(D^d))$. Arguing almost as in the proof of Lemma~2.3 of \cite{BBFS}, one shows that $\hat{\varrho}$ is homotopic relative to $\partial D^d$ to an admissible map $\overline{\varrho}$ with $\avol(\overline{\varrho})\leq c\cdot \vol^d(\varrho)$ through a Lipschitz homotopy of zero volume and with image in $\supp(\hat{\varrho}(D^d))$. This proves the lemma.
\end{proof}

\begin{proof}[Proof of Lemma~\ref{lem:quant-admissible}]
 Denote by $N$ the dimension of $Y$. Let $\psi\from Y\to Z$ be as in the statement of the lemma.
 % We will construct $\overline{\psi}$ one skeleton at a time.
 For $0\leq d\leq N$, let $Y_d$ be the mapping cylinder of the inclusion $Y^{(d)}\subset Y$, i.e.
 $$Y_d = (Y\times\{0\})\cup (Y^{(d)}\times [0,1])\subset Y\times[0,1].$$
 We will construct $\overline{\psi}$ by constructing a sequence of auxiliary maps $\psi_d\from Y_d\to Z$ such that for all $d$ and all $y\in Y$, $\psi_d(y,0)=\psi(y)$, $\psi_{d+1}$ extends $\psi_d$, and $\overline{\psi|_{Y^{(d)}}}:=\psi_d(\cdot,1)$ satisfies the conditions of the lemma. Then $\overline{\psi}=\psi_N(\cdot,1)$ is the desired map.
 
 % satisfies the conditions of the lemma. That is, for any simplex $\sigma\in \cF(Y^{(d)})$, $\psi_d|_{\sigma\times\{1\}}$ is admissible, $\psi_d(\sigma\times \{1\}) \subset \supp \psi(\sigma)$,
 %$$\avol(\psi_d|_{\sigma\times\{1\}}) \le b \cdot \avol(\bar{\psi}|_\sigma) \le b\cdot \sum_{\sigma'\in\cF(\sigma)}\vol^m(\psi|_{\sigma'}),$$
 %and if $\sigma$ is an edge, then $\psi_d|_{\sigma\times\{1\}}$ is a constant-speed edge path.
 
 Let $\psi_0$ be the Lipschitz map that coincides with $\psi$ on $Y\times\{0\}$ and such that for $v\in\cF^0(Y)$ the following holds: $\psi_0(v,1)$ is a point in $Z^{(0)}$ which is closest to $\psi_0(v,0)=\psi(v)$ and $\psi_0(v,\cdot)$ is the constant speed para\-metrization of the segment between its endpoints. Since $\psi_0(v,1)$ and $\psi(v)$ lie in the same simplex of $Z$, $\overline{\psi|_{Y^{(0)}}}:=\psi_0(\cdot,1)$ satisfies the lemma and each segment $\psi_0(v\times [0,1])$ has length at most $1$.
 
 We use Lemma~\ref{lem:def-to-admissible} to extend $\psi_0$ to a map $\psi_1$ as follows. Let $e\in\cF^1(Y)$ be an edge. Then $(e\times\{0\})\cup (\partial e \times[0,1])$ is a path in $Y_0$, and $$\alpha_e=\psi_0|_{(e\times\{0\})\cup (\partial e \times[0,1])}$$
 is a path in $Z$. By Lemma~\ref{lem:def-to-admissible}, we can define $\psi_1$ on $e\times [0,1]$ so that $\psi_1|_{e\times \{1\}}$ is a constant-speed edge path approximating $\alpha_e$ and so that $\psi_1|_{e\times [0,1]}$ is a homotopy between $\alpha_e$ and its approximation. We obtain $\psi_1$ by repeating this process for every edge of $Y$.
 
 Then $\psi_1$ is a Lipschitz map extending $\psi_0$ and it has the following properties. For each edge $e\in \cF^1(Y)$, $\psi_1|{e\times \{1\}}$ is an edge path para\-metrized with constant speed. Moreover, the length of $\psi_1|_{e\times\{1\}}$ and the $2$--volume of $\psi_1|_{e\times[0,1]}$ are both bounded from above by $c\cdot \length(\alpha_e) \le c\cdot \length(\psi|_e) + 2c$ for some $c=c(\dim Z)$. Finally, $\psi_1(e\times[0,1])\subset\supp(\psi(e))$, so $\overline{\psi|_{Y^{(1)}}}:=\psi_1(\cdot,1)$ satisfies the lemma.
 
 We repeat the process to obtain a Lipschitz map $\psi_d\from Y_d\to Z$ for each $d=2,3,\dots, N$. That is, we suppose that $\psi_{d-1}$ is already defined and for each $\sigma\in\cF^d(Y)$, we define 
 $$\alpha_\sigma=\psi_{d-1}|_{(\sigma\times\{0\})\cup (\partial \sigma \times[0,1])}.$$
 This is a Lipschitz map from a $d$--disc to $Z$. By Lemma~\ref{lem:def-to-admissible}, we can define $\psi_d$ on $\sigma \times [0,1]$ so that $\psi_d|_{\sigma \times \{1\}}$ is an admissible map approximating $\alpha_\sigma$ and so that $\psi_d|_{\sigma \times [0,1]}$ is the homotopy between $\alpha_\sigma$ and its approximation. 

 Then $\psi_d$ extends $\psi_{d-1}$ and has the following properties. For every $\sigma\in\cF^d(Y)$ the restriction $\psi_d|_{\sigma\times\{1\}}$ is admissible and $\psi_d(\sigma\times[0,1])\subset \supp(\psi(\sigma))$. Moreover, both $\avol(\psi_d|_{\sigma\times\{1\}})$ and $\vol^{d+1}(\psi_d|_{\sigma\times[0,1]})$ are bounded from above by 
 $$c\cdot \vol^{d}(\alpha_\sigma) \le c\cdot\vol^d(\psi|_{\sigma}) + c\cdot \vol^d(\psi_{d-1}|_{\partial \sigma\times[0,1]}) \le b_d \sum_{\sigma'\in \cF(\sigma)} \vol^{\dim(\sigma')}(\psi|_{\sigma'})$$
 for some $b_d>0$ depending on $d$ and the dimension of $Z$. That is, $\overline{\psi|_{Y^{(d)}}}:=\psi_d(\cdot,1)$ satisfies the lemma for some constant $b_N$ depending on the dimensions of $Y$ and $Z$, and we define $\overline{\psi}:=\psi_{N}(\cdot,1)$. 
 In fact, if $\sigma\in \cF^d(Y)$ and $d>\dim Z$, then $\avol(\overline{\psi}|_\sigma)=0$, so $\overline{\psi}$ satisfies the lemma for $b=b_{\dim Z}$.
 
If $\psi$ is equivariant then we can take all the maps $\psi_d$ to be equivariant as well, and hence also $\overline{\psi}$.
\end{proof}

In preparation for the proof of Lemma~\ref{lem:normalizing uncollapsed} we first establish the following results.

\begin{lemma}\label{lem:collared admissible}
  Let $Z$ be a finite dimensional simplicial complex. If $f\from D^d\to Z$ is an admissible map with uncollapsed balls $B_1,\dots, B_k$, it is homotopic to an admissible map $f'\from D^d\to Z$ such that $f$ and $f'$ agree on $D^d\setminus \bigcup B_i$, $\avol(f)=\avol(f')$, and each uncollapsed ball $B_i'$ of $f'$ is a collared ball contained in $B_i$.
\end{lemma}
\begin{proof}
  Each $B_i$ is homeomorphic to a closed ball.  Let $B'_i\subset B_i$ be the image of the ball $B(0,\frac{1}{2})\subset\R^d$ under a homeomorphism from $B(0,1)\subset\R^d$ to $B_i$.  Then there is a map $\zeta_i\from B_i\to B_i$ that sends $B'_i$ homeomorphically to $B_i$, such that $\zeta_i(B_i\setminus B'_i)\subset \partial B_i$ and $\zeta_i$ is the identity on $\partial B_i$.  Let $f'\from D^d\to Z$ be the map such that $f'$ agrees with $f$ outside the uncollapsed balls $B_i$ and such that $f'(x)=f(\zeta_i(x))$ for all $x\in B_i$.  This is an admissible map whose uncollapsed balls are the collared balls $B'_1,\dots, B'_k$.
\end{proof}

\begin{lemma}\label{lem:connected sums}
  Let $Z$ be a finite dimensional simplicial complex and let $f\from D^d\to Z$ be an admissible map with uncollapsed balls $B_1,\dots, B_k$.  Let $\sigma_i=f(B_i)\in \cF^d(Z)$ for all $i$.  Let $C_1,\dots, C_k\subset D^d$ be a set of disjoint collared balls, and for $i=1,\dots, k$, let $\varphi_i\from C_i\to \sigma_i$ be a homeomorphism such that $(f|_{B_i})^{-1} \circ \varphi_i\from C_i \to B_i$ preserves orientation.  Then there is an admissible map $g\from D^d\to Z$ with uncollapsed balls $C_1,\dots, C_k$ such that $g|_{C_i}=\varphi_i$ and $g|_{\partial D^d}=f|_{\partial D^d}$.  Furthermore, there is a homotopy between $f$ and $g$ that fixes $\partial D^d$ pointwise and has image in $\supp f(D^{d})$. If $f|_{\partial D^d}$ and $\varphi_i$ are Lipschitz then we can take $g$ to be Lipschitz as well.
\end{lemma}

\begin{proof}
  By Lemma~\ref{lem:collared admissible}, we may suppose that the $B_i$ are disjoint collared balls.  By the uniqueness of connected sums, the complements $J=D^d \setminus \bigcup_i \inter(B_i)$ and $K=D^d \setminus \bigcup_i \inter(C_i)$ are homeomorphic; indeed, any collection of orientation-preserving homeomorphisms $h_i\from \partial C_i\to \partial B_i$ can be extended to a homeomorphism $h\from K\to J$ that fixes $\partial D^d$ pointwise.  Thus, there exists a homeomorphism $h\from D^d\to D^d$ which fixes $\partial D^d$ pointwise and satisfies $h|_{C_i}=(f|_{B_i})^{-1} \circ \varphi_i$ for all $i=1,\dots, k$. Let $g=f\circ h$.  This map has the desired restrictions.  Since $h$ is homotopic to $\id_{D^d}$, $g$ is homotopic to $f$.
  
  Suppose now that $f|_{\partial D^d}$ and the $\varphi_i$ are Lipschitz. Then $g(D^d)\subset f(D^d)\subset \supp f(D^d)$ and $g(K)\subset Z^{(d-1)}$. Since $\supp g(K)\subset Z^{(d-1)}$ is locally Lipschitz $n$--connected for every $n$ we can use the Lipschitz extension theorems in \cite{Alm62} or \cite{Hoh93} to approximate $g$ arbitrarily closely by an admissible Lipschitz map $g'$ which still satisfies $g'|_{C_i}= \varphi_i$ and $g'|_{\partial D^d} = f|_{\partial D^d}$ and is still homotopic to $f$ via a homotopy that fixes $\partial D^d$ pointwise.  
  % citation?
\end{proof}

We are ready to prove Lemma~\ref{lem:normalizing uncollapsed}.

\begin{proof}[Proof of Lemma~\ref{lem:normalizing uncollapsed}]
 We construct a homotopy between $f$ and a suitable map $h$ satisfying the properties of the lemma by repeatedly applying Lemma~\ref{lem:connected sums} as follows.
 Let $N$ be the dimension of $Y$. For $1\leq d\leq N$ set 
 $$Y_d = (Y\times\{0\})\cup (Y^{(d)}\times[0,1])\subset Y\times[0,1].$$ 
 We will define maps $H_d\from Y_d\to Z$ such that $H_{d+1}$ extends $H_d$ and such that $H_N$ is the desired homotopy. Let $H_1$ be the map such that $H_1(y,0) = f(y)$ for all $y\in Y$ and $H_1(y,t)= f(y)$ for all $(y,t)\in Y^{(1)}\times[0,1]$. 
 
 Let $2\le d\le N$ and suppose that we have already defined $H_{d-1}$ with $H_{d-1}|_{\sigma\times \{1\}}$ admissible and $H_{d-1}(\sigma\times[0,1])\subset\supp f(\sigma)$ for all $\sigma\in\cF^{d-1}(Y)$. Let  $\sigma\in\cF^d(Y)$. Then $H_{d-1}(\partial\sigma \times [0,1])\subset\supp f(\partial\sigma)\subset Z^{(d-1)}$ and thus $H_{d-1}$ is admissible on the $d$--disc $(\sigma\times\{0\})\cup(\partial\sigma\times[0,1])$.
 By Lemma~\ref{lem:connected sums} there exists a continuous map $H_{d}|_{\sigma\times[0,1]}$ which coincides with $H_{d-1}$ on 
 $(\sigma\times\{0\})\cup(\partial\sigma \times[0,1])$ and such that $H_d|_{\sigma\times\{1\}}$ satisfies Lemma~\ref{lem:normalizing uncollapsed}. That is, $H_d|_{\sigma\times\{1\}}$ is admissible and its uncollapsed balls are the $C_i^\sigma$'s and for each $i$ it agrees with $g_i^\sigma$ on $C_i^\sigma$. Moreover, $H_d(\sigma\times[0,1])\subset \supp f(\sigma)$. Since $\sigma$ was arbitrary this defines a map $H_d$ on all of $Y_d$.
 
 Applying this construction repeatedly we obtain maps $H_d\from Y_d\to Z$ for $d=2,\dots, N$. The map $H=H_N\from Y\times[0,1]\to Z$ is a homotopy from $f$ to the map $h\colon Y\to Z$ given by $h=H_N(\cdot, 1)$ and it follows from the construction that $h$ has the properties asserted in the statement of the lemma.
\end{proof}

\bibliography{bibli-Hoelder-Carnot}
\end{document}